\documentclass[11pt,reqno]{amsart}
\usepackage{amsfonts, amssymb, amsmath, dsfont, graphicx, subcaption, enumerate, fullpage, xcolor}

\numberwithin{equation}{section}

\DeclareMathOperator{\E}{\mathbb{E}}

\DeclareMathOperator{\rad}{rad}
\DeclareMathOperator{\Id}{Id}

\renewcommand{\Im}{Im}

\renewcommand{\P}{\mathbb{P}}
\newcommand{\R}{\mathbb{R}}

\newcommand{\EE}{\mathcal{E}}

\newcommand{\Ac}{A_S^{\dagger}}

 \newtheorem{theorem}{Theorem}[section]
\newtheorem{proposition}[theorem]{Proposition}

\newtheorem{lemma}[theorem]{Lemma}

\newtheorem{definition}[theorem]{Definition}

\theoremstyle{remark}
\newtheorem{remark}[theorem]{Remark}

\begin{document}

\bibliographystyle{abbrv}

\title{On block Gaussian sketching for the Kaczmarz method}
\author{Elizaveta Rebrova \and Deanna Needell}

\address{Department of Mathematics, University of California - Los Angeles, 520 Portola Plaza, Los Angeles, CA 90095}
\email{rebrova@math.ucla.edu, deanna@math.ucla.edu}

\maketitle

\begin{abstract}
The Kaczmarz algorithm is one of the most popular methods for solving large-scale over-determined linear systems due to its simplicity and computational efficiency. This method can be viewed as a special instance of a more general class of sketch and project methods. Recently, a block Gaussian version was proposed that uses a block Gaussian sketch, enjoying the regularization properties of Gaussian sketching, combined with the acceleration of the block variants. Theoretical analysis was only provided for the non-block version of the Gaussian sketch method. 

Here, we provide theoretical guarantees for the block Gaussian Kaczmarz method, proving a number of convergence results showing convergence to the solution exponentially fast in expectation. On the flip side, with this theory and extensive experimental support, we observe that the numerical complexity of each iteration typically makes this method inferior to other iterative projection methods. We highlight only one setting in which it may be advantageous, namely when the regularizing effect is used to reduce variance in the iterates under certain noise models and convergence for some particular matrix constructions.

\end{abstract}

\section{Introduction}
The main goal of this paper is to study a proposed block Gaussian Kaczmarz variant, both theoretically and experimentally, for solving large-scale linear systems. We start with a brief description of the relevant members of the rich family of Kaczmarz iterative methods.

\subsection{Kaczmarz and randomized Kaczmarz algorithms}\label{intro1}
The Kaczmarz method \cite{Kaz} is an iterative method for solving large-scale (typically highly over-determined) linear systems. Being simple, efficient and well-adapted to large amounts of data (due to its iterative nature), the Kaczmarz method is widely used in a variety of applications, from image reconstruction to signal processing \cite{SesSta,Nat,Feinco,herman1993algebraic}. Given a consistent  system (we will consider inconsistent systems later) of linear equations of the form
\begin{equation}\label{main_system}
Ax = b,
\end{equation}
the original Kaczmarz method starts with some initial guess $x_0  \in \R^n$, and then iteratively projects the previous approximation $x_k$ onto the solution space of the next equation in the system. Namely, if $A_1, \ldots, A_m \in \R^n$ are the row vectors of $A$, then the $k$-th step of the algorithm is given by:
\begin{equation}\label{standardRK}
x_{k} = x_{k-1}+ \frac{b_i - A_i^T x_{k-1}}{\|A_i\|^2} A_i,
\end{equation}
where $b = (b_1, \ldots, b_n) \in \R^m$ is the right hand side of the system, $i =  k\mod m$ and $x_{k-1} \in \R^n$ is the approximation of a solution $x_*$ obtained in the previous step. The process continues until it triggers an appropriate convergence criterion.

To provide theoretical guarantees for the convergence of the method, Strohmer and Vershynin \cite{StrVer} proposed to choose the next row $A_i$ at random with probability proportional to the $L_2$ norm of the row $A_i$. The authors have shown that this \emph{randomized Kaczmarz algorithm} is guaranteed to converge exponentially in expectation, namely,
\begin{equation}\label{standard_convergence_rate}
\E \|x_k - x_*\|^2_2 \le \left(1 - \frac{1}{R}\right)^k \|x_0 - x_*\|^2_2,
\end{equation}
where $x_*$ is the solution of the system \eqref{main_system} and $R$ is a constant depending only on the matrix $A$, namely, $R = \|A\|^2_F/\sigma_{min}^2(A)$. 

There is a variety of extensions and refinements of the first randomized Kaczmarz method. They include specializations of the method to some other classes of problems (like solving inconsistent linear systems \cite{Nee}, phase retrieval \cite{TanVer}, stochastic gradient descent \cite{NeeWarSre}, etc);  improvements in the weighting of the rows (from the one based on $\|A_i\|_2$ to some better ``optimal" probabilities, see, e.g., \cite{GowRic}), and new hybrid methods based on Kaczmarz \cite{LoeHadNee}. We omit a detailed discussion of such related work but refer the reader to those mentioned and others therein.

\subsection{Notations} Here and further, we denote by $\sigma_{min}(A)$ and $\sigma_{max}(A)$ the smallest and largest singular values of the matrix $A$ (that is, eigenvalues of the matrix $\sqrt{A^TA}$). Then,  $\|A\|_F :=\sqrt{trace(A^TA)}$ (Frobenius, or Hilbert-Shmidt, norm of the matrix) and $\|A\| := \sup_{\|x\|_2 = 1} \|Ax\|_2$ (operator norm of the matrix). Moreover, we always assume that the matrix $A$ has full column rank, so that $\sigma_{min}(A) > 0$ and the convergence rate is non-trivial.

\subsection{Organization and Contribution} 

The remainder of the paper is organized as follows. Next, in Section \ref{intro2} we describe the block Gaussian Kaczmarz method, which can be viewed as a block Kaczmarz method (that utilizes many rows for each projection) preprocessed in each iteration via a random Gaussian matrix (the \textit{sketch}). We state our main results for its convergence as well as some extensions and implementation variations in Subsections \ref{intro3} and \ref{intro4}. To the best of our knowledge, our results are the first theoretical guarantees for the block Gaussian sketch variant of the Kaczmarz method recently described in \cite{GowRic}. We present the proofs of these results in Section \ref{sec:proofs}. We present experimental results in Section \ref{experiments} that showcase four main observations. First, for consistent systems, larger block sizes tend to yield faster convergence; in fact a large block size of $s=n$ is computationally optimal (when computationally possible) and results in convergence in a single step. Secondly, however, it seems that the benefit of the block Gaussian Kaczmarz method typically stems from the fact it is a block variant -- and not that it uses Gaussian sketching. Thus, the (non-Gaussian) block Kaczmarz method would be typically preferred since it doesn't require computationally heavy sketching. We do however showcase a carefully constructed matrix model for which Gaussian sketching does become advantageous over its non-Gaussian counterparts. Third, we demonstrate that the block Gaussian Kaczmarz (BGK) method can be implemented using a finite collection of sketches, and lastly we show that for inconsistent systems, the BGK method offers advantages in terms of variance reduction in the solution error.

\section{Block Gaussian Kaczmarz} \label{intro2}

The extension that will be of our major interest throughout the paper is a version of the Kaczmarz algorithm that uses blocks of the rows for iterative projections (rather than individual rows). Namely, the $(k+1)$-st iteration has the form
\begin{equation}\label{blockRK}
x_{k+1} = x_k + (A_{\tau})^\dagger (b_{\tau} - A_{\tau} x_k),
\end{equation}
where $A_{\tau}$ and $b_{\tau}$ denote the restriction onto the (row) indices from the subset $\tau \subset \{1, \ldots, m\}$ and $(A_{\tau})^\dagger$ denotes the Moore-Penrose inverse of the matrix $A_{\tau}$.

This framework was initially proposed by Elfving \cite{Elf}, and its randomized version was presented and analyzed in the paper by Needell and Tropp \cite{NeeTro}. In the randomized version, the matrix $A$ is split into several row blocks, and at each iteration one of these blocks is chosen uniformly at random with replacement. The authors prove the exponential convergence of the method with a strong convergence constant,
\begin{equation}\label{block_kaz}
\E \|x_k - x_*\|^2_2 \le \left(1 - \frac{\sigma_{min}^2(A)}{C \|A\|^2 \log(m + 1)}\right)^k \|x_0 - x_*\|^2_2,
\end{equation}
if we manage to choose a ``good" row block partition, and under an assumption that all the rows are standardized, namely, $\|A_i\|_2 = 1$; see \cite{NeeTro} for details.

Although the existence of this ``good" partition is theoretically guaranteed, it is not always straightforward how to find such partition (e.g., if $A$ has coherent rows). However, experimental evidence shows that the block Kaczmarz method still exhibits fast convergence  even in these cases. This observation is especially interesting since coherent matrices are precisely the examples for which standard randomized Kaczmarz does not perform well (as projections at each step follow roughly the same direction, which might not be a direction toward the true solution $x_*$). Some theoretical analysis of this improvement for blocks of size two is available in \cite{NeeWar}. 

\bigskip
A unified view on both regular and block Kaczmarz methods, along with many other randomized iterative solvers, was proposed by Gower and Richt\'arik in \cite{GowRic}. The main idea of their \emph{sketch-and-project} framework is the following. One can observe that the random selection of a row (or a row block) can be represented as a \textit{sketch}, that is, left multiplication by a (random or deterministic) vector (or a matrix). This results in a preprocessing of every iteration of the method, which is represented by a projection onto the image of the sketch. Sketching preprocessing is a valuable procedure on more that just the Kaczmarz method (as proposed in \cite{GowRic,gower2015stochastic,loizou2017momentum,loizou2019convergence}). 
 
Thus, in the case of the Kaczmarz method, the iteration can be written as
\begin{equation}\label{iter}
x_{k+1} = (\Id - (S^TA)^{\dagger}S^TA)x_k + (S^TA)^{\dagger} S^Tb,
\end{equation}
where $S$ is the sketch matrix, taken from some (typically random) matrix model at each step. Although $S$ will be drawn in each iteration (later we discuss variants to this implementation), we omit an iteration index for notational simplicity. For brevity, we will denote $A_S :=S^TA$.

Clearly, in the case of block Kaczmarz \eqref{blockRK}, sketch matrices $S$ are just shifted identity matrices tabbed by zeros for the correct size ($m$ by block size). Standard Kaczmarz \eqref{standardRK} can be, of course, considered as a special case of a block method with block size one.

The sketch-and-project viewpoint suggests a natural idea to generalize the methods by adopting some other sketch matrices $S$. Gower and Richt\'arik propose to take $S$ to be a standard Gaussian matrix with independent entries. The authors show exponential convergence with the standard rate \eqref{standard_convergence_rate} with $R = 2\|A\|^2_F/\pi \sigma_{min}^2(A)$ in the one dimensional case (when $S$ is a Gaussian vector in $\R^m$). To our best knowledge, theoretical analysis for the general block case has not previously been established, nor have extensive experimental tests been performed. We refer to \eqref{iter} with Gaussian matrix sketches $S$ as the \textit{block Gaussian Kaczmarz} or \textit{BGK} method.

\subsection{Main results} \label{intro3}
Here, we continue the study of  Gaussian sketch matrices in application to the Kaczmarz methods; developing the theoretical analysis beyond the vectorized version to the general block Gaussian Kaczmarz method. The proofs of these results are presented in Section \ref{sec:proofs}. Our first convergence result is as follows:

\begin{theorem}\label{theor1}
Suppose $A$ is a $m \times n$ matrix with full column rank ($m \ge n$), such that its condition number $\kappa^2(A) := \sigma^2_{max}(A)/\sigma^2_{min}(A) \le e^{m/4}/3$, and let $x_*$ be a solution of the system $Ax = b$. For any initial estimate $x_0$, the BGK method (iteration \eqref{iter} with $S$ being an $m \times s$ random matrix with i.i.d. standard normal entries) produces a sequence $\{x_k, k \ge 0\}$ of iterates that satisfy
\begin{align}\label{c_rate_1}
\E\|x_{k} - x_*\|_2^2 \le \left( 1 - \frac{s}{15m \kappa^2(A)}\right)^k \|x_0- x_*\|_2^2.
\end{align}
\end{theorem}
\begin{remark}
Note that the condition $\kappa^2(A) := \sigma^2_{max}(A)/\sigma^2_{min}(A) \le e^{m/4}/3$ naturally holds for many standard classes of matrices. For example, random Gaussian matrices have condition numbers $\kappa(A) \sim m$ (\cite{edelman1989eigenvalues,Sh1}). The same holds for the broader class of random matrices with i.i.d. elements having sub-gaussian tails (\cite{LittleOffInv}). Moreover, heavy-tailed models, when the matrix entries have only two finite moments, still have polynomial condition numbers $\kappa(A)$ with high probability \cite{RebTik}.
\end{remark}
An alternative (although very similar) estimate can be obtained for all matrices, without a condition number assumption, in trade of the absolute constants:

\begin{theorem}\label{theor2}
Suppose $A$ is a $m \times n$ matrix with full column rank ($m \ge n$) and let $x_*$ be a solution of the system $Ax = b$. For any initial estimate $x_0$, the BGK method (iteration \eqref{iter} with $S$ being an $m \times s$ random matrix with i.i.d. standard normal entries) produces a sequence $\{x_k, k \ge 0\}$ of iterates that satisfy
\begin{align}\label{c_rate_2}
\E\|x_{k} - x_*\|_2^2 \le \left( 1 - \frac{1}{80}\left[\frac{\sqrt{s}\sigma_{min}(A)}{\sqrt{s} \|A\| + \|A\|_F}\right]^2\right)^k \|x_0- x_*\|_2^2.
\end{align}
Here, $C > 0$ is an absolute constant.
\end{theorem}

We note that Theorem~\ref{theor2} is usually stronger than Theorem~\ref{theor1}; since $\|A\|_F \le m \|A\|$ for any matrix, Theorem~\ref{theor1} might give tighter results than Theorem~\ref{theor2} for some matrices $A$ with $\|A\|_F \approx m \|A\|$, but its advantage will be at most by a constant multiple in the convergence rate $R$.

We also note that in the case $s = 1$ we recover the convergence rate $c \sigma_{min}^2(A)/\|A\|_F^2$ that was proved earlier in \cite{GowRic}. We also obtain theoretical evidence for the observed convergence speed-up with bigger sketch sizes $s > 1$. Indeed, from Theorems~\ref{theor1} and \ref{theor2} we can see that the expected gain from using block sketches is linear in the size of the block $s$ (if we look at the per-iteration gain). The same advantage is experimentally observed for the (non-Gaussian) block Kaczmarz methods (see e.g. \cite{NeeTro,NeeWar} and below in Section \ref{experiments}), however, prior theoretical analysis, such as the rate in \eqref{block_kaz}, did not allow us to trace the dependence on $s$.
Of course, for larger $s$, both the sketching step (computing $S^TA$) and the inversion step (computing $(S^TA)^{\dagger}$) become slower. In Section~\ref{experiments} we study this tradeoff numerically. 

\subsection{Sampling from a finite collection} \label{intro4}
As presented, the BGK method \eqref{blockRK} requires a new random matrix $S$ to be drawn in each iteration. 
A natural question is whether we actually need to generate a new Gaussian matrix at each step. In some settings, if memory is not an issue, we may rather have a finite set of matrices (potentially, generated in advance) and sample from it. For comparison, in the standard Kaczmarz methods viewed as sketched versions, we have only a finite set of  sketch matrices (and the cardinality of this set is the number of rows of the matrix divided by the number of blocks). We prove that in the Gaussian case we can be satisfied with a finite collection of sketches as well. We call this variant of the method the \textit{finite} block Gaussian Kaczmarz method. We establish the following result in this setting.

\begin{theorem}[Finite collection]\label{theor3} Suppose $A$ is a $m \times n$ matrix with full column rank ($m \ge n$) and let $x_*$ be a solution of the system $Ax = b$. Let $ N$ be such that $64c m^2 \log m \le N \le \exp(m/3)$ (for some $c > 3$). Let $\mathcal{S} = \{S^{(1)}, \ldots, S^{(N)}\}$ be a random set of $m \times s$ random matrices with i.i.d. standard normal entries. Then, with probability at least $1 - 1.1m^{3-c}$, for any initial estimate $x_0$, finite BGK method (iteration \eqref{iter} with $S$ being chosen randomly with replacement\footnote{Sampling is called with replacement when a unit selected at random from the collection is returned to the collection and then a second element is selected at random. So, the same element may be selected more than once.} from the set $\mathcal{S}$) produces a sequence $\{x_k, k \ge 0\}$ of iterates that satisfy
\begin{align*}
\E\|x_{k} - x_*\|_2^2 \le \left( 1 - \frac{s}{36m \kappa^2(A)}\right)^k \|x_0- x_*\|_2^2.
\end{align*}
\end{theorem}
Thus, the convergence rate is as good as in the case of  taking a new sketch at each iteration (Theorem~\ref{theor2}). However, the size of a pre-selected set $\mathcal{S}$  required by Theorem~\ref{theor3} ($N \gg m^2 \log m$) is likely too big to be practical. Our experiments show that in practice the size $N \sim m/s$ (number of rows divided by the block size, like in the regular block Kaczmarz case) is enough to demonstrate the same convergence (see Section \ref{experiments}.

\bigskip

\subsection{Remarks on the inconsistent case}
So far we considered consistent systems with a unique solution $x_*$ such that $Ax_* = b$. However, for inconsistent systems one may wish to solve the overdetermined least-squares problem
\begin{equation}\label{minimization_problem}
  \text{ minimize}_x \|Ax - b\|_2^2
 \end{equation}
for the unique minimizer $x_*$. Comparing to the consistent case, the noise in the system $e := Ax_* - b$ is known to incur a second ``error" term (the so-called \textit{convergence horizon}), so that
\begin{equation}\label{noisy_conversgence_rate}
\E \|x_k - x_*\|_2^2 \le \beta \|x_k - x_*\|_2^2 + \varphi,
\end{equation}
where $\beta \in (0,1)$ and  $\varphi = \varphi(e)$ does not decrease with iteration and gives the radius of the ball that will eventually contain the final iterates \cite{NeeTro,Nee}. The exponential convergence rate $\beta$ is the same as that of the corresponding method in the consistent case, and the convergence horizon $\varphi(e) = \|e\|^2_2/\sigma_{min}^2(A)$ was established for the block Kaczmarz method in \cite{NeeTro}, improving the weaker bound $\varphi(e) = n\|e\|^2_{\infty}/\sigma_{min}^2(A)$ known for standard Kaczmarz \cite{Nee}. 

 In the BGK case, we can also prove the expected convergence rate in the form \eqref{noisy_conversgence_rate}. 
 The following result is proved in Section~\ref{proof_with_error}.
\begin{theorem}[Inconsistent case, finite sample collection]\label{theor4}
Suppose $A$ is a $m \times n$ matrix with full column rank ($m \ge n$), let $x_*$ be a solution of the quadratic minimization problem~\eqref{minimization_problem}, and $e := Ax_* - b$. Let $ N$ be such that $64c m^2 \log m \le N \le \exp((\sqrt{n} - \sqrt{s})^2/16)$ (for some $c > 3$). Let $\mathcal{S} = \{S^{(1)}, \ldots, S^{(N)}\}$ be a set of $m \times s$ random matrices with i.i.d. standard normal entries. Then, with probability at least $1 - 1.1m^{3-c} - e^{-(\sqrt{n} - \sqrt{s})^2/16}$, for any initial estimate $x_0$, the finite BGK method (iteration \eqref{iter} with $S$ being chosen randomly with replacement from the set $\mathcal{S}$) produces a sequence $\{x_k, k \ge 0\}$ of iterates that satisfy
\begin{equation}\label{c_rate_noisy_fin}
\E\|x_{k} - x_*\|_2^2 
\le \left( 1 -  \frac{s}{36m \kappa^2(A)}\right)^k \|x_0- x_*\|_2^2 
+ \frac{Cm\kappa^2(A)}{\sigma_{min}^2(A)} \cdot \frac{\|e\|^2_2}{(\sqrt{n} - \sqrt{s})^2},
\end{equation}
where $C > 0$ is an absolute constant.
\end{theorem}
\begin{remark} Both in probability and in the size of the collection condition we can observe that we need $n \gg s$ for this result. Typically, it is enough to have $n = \alpha s$ where $\alpha  \in [0,1]$ is a constant and $n$ is large enough. One can use more sophisticated estimates in place of~\eqref{smin_gaus} (such as, \cite[Corollary V.2.1]{feldheim2010universality}) to get better probability and allow a bigger collection, but they still depend on $(\sqrt{n} - \sqrt{s})$ being large.
\end{remark}

A similar result can be proved in the framework where we redraw an independent Gaussian sketch matrix $S$ at every iteration of the method rather than use a finite collection.
\begin{theorem}[Inconsistent case, independent Gaussian sampling]\label{theor5}
Suppose $A$ is a $m \times n$ matrix with full column rank ($m \ge n$), let $x_*$ be a solution of the quadratic minimization problem~\eqref{minimization_problem}, and $e := Ax_* - b$.  Assume that $s \le \alpha n$ for some constant $\alpha < 1$. For any initial estimate $x_0$, the BGK method (iteration \eqref{iter} with $S$ being an $m \times s$ random matrix with i.i.d. standard normal entries) produces a sequence $\{x_k, k \ge 0\}$ of iterates that satisfy
\begin{equation}\label{c_rate_noisy}
\E\|x_{k} - x_*\|_2^2
\le  \left( 1 - \frac{1}{80}\left[\frac{\sqrt{s}\sigma_{min}(A)}{\sqrt{s} \|A\| + \|A\|_F}\right]^2\right)^k \|x_0- x_*\|_2^2 
+ \frac{C(\sqrt{s} \|A\| + \|A\|_F)^2\|e\|^2_2}{\sigma_{min}^4(A)(\sqrt{n} - \sqrt{s})^2},
\end{equation}
where $C > 0$ is an absolute constant.
\end{theorem}

Note that if $s$ is small and $A$ is well-conditioned (so that $\|A\|_F^2 \sim n \sigma_{min}^2(A)$), then \eqref{c_rate_noisy} recovers the convergence horizon of the standard Block Kaczmarz $O(\|e\|^2/\sigma_{min}^2(A))$ and the first term gives the optimal rate obtained in Theorem~\ref{theor2}. 

\begin{remark}[Absolute constants]
We obtain the result of Theorem~\ref{theor4} with the constant $C = 300$ and Theorem~\ref{theor4} with the constant $C = 1600$. These constant values are not optimized (as well as some other constants throughout the text).
\end{remark}

\begin{remark}[Dependence between the dimension and sketch size]
The performance of BGK on consistent systems improves with increasing sketch size (both in iteration and in time), but for inconsistent systems the sketch size vs convergence speed trade-off becomes more interesting: observe that for $s < n$, an increasing block size $s$ improves the first term in~\eqref{c_rate_noisy_fin}, but makes the dependence on the convergence horizon (the second term) worse. Our experiments (see Figure~\ref{fig:noisy_blocks}) suggest that, like in the consistent case, convergence in iteration improves with increasing $s$ (until the regime $s \approx n$), but with regard to computational time one might prefer smaller block sizes, especially if the goal is to achieve smaller approximation error. 

Furthermore, for consistent systems, taking a sketch of  size $s = n$ one would solve the system exactly in one step. In the presence of error, the case $s = n$ is a very bad choice since it makes the error term diverge. 

Increasing $s$ in the range $s > n$ in the consistent case does not make sense: it preserves the one step convergence, but makes the iteration slower. In the inconsistent case, further increase in $s$ improves the second error term (keeping the first term zero after one iteration), so, potentially, taking very big $s$ is beneficial for the convergence per-iteration rate. However, because of the increasing iteration complexity and memory required to perform it, taking $s > n$ still does not seem practical. 
See Section~\ref{experiments} for more details and experimental evidence.
\end{remark}

\begin{remark}[Noise term is bounded with high probability] \label{high_prob_remark}
Let $S$ be an $m \times s$ random sketch matrix with i.i.d. standard normal entries, matrix $A \in \R^{m \times n}$ and noise vector $e \in \R^{m}$ are defined as above. Assume that $s = \alpha n$ for some constant $\alpha> 0$. Then $\P(\EE_S )\ge 1 - e^{-cn}$, where
\begin{equation}\label{components}
\EE_S:= \left\{\|A_S^\dagger S^Te\|_2^2 \le \frac{8s\|e\|^2_2}{(\sqrt{n} - \sqrt{s})^2 \sigma_{min}^2(A)}\right\} \end{equation}
and $c = c(\alpha)$ is a positive constant. Indeed, this follows from Lemma~\ref{independence_lemma} below, the standard estimate for the smallest singular value of a Gaussian matrix in the form of~\eqref{smin_gaus} (with $t = (\sqrt{n} - \sqrt{s})/2$) and Bernstein's inequality. The estimate~\eqref{components} shows that we could make $O(e^{cn})$ steps of the BGK algorithm so that all the error terms stay bounded with high probability. This explains an experimentally observed robustness of BGK (see Figure~\ref{fig:spiky_noise}). In the case when the vector $e$ is sparse and spiky, Gaussian sketching ``smoothes the error", namely, some iterations of the regular Block Kaczmarz method result in a huge error $\|x_k - x_*\|$, whereas the iterations of BGK steadily converge towards the limiting error level.
\end{remark}

\section{Proofs of Main Results}\label{sec:proofs}

\subsection{Proof of Theorem~\ref{theor1}, convergence estimate via condition number}\label{proof1}
We start the proof with two auxiliary lemmas. The first one provides a sharp non-asymptotic bound for the norm of a random Gaussian matrix with independent $N(0,1)$ entries. 

\begin{lemma}\label{sup_norm}
Let $X$ be a $m \times n$ random matrix, $m \ge n$, whose entries are independent copies of a standard normal random variable. Then for all $t \ge 0$
$$
\P(\sigma_{max}(X) > (2+t) \sqrt{m}) \le \exp(-t^2m/2).
$$
\end{lemma}
This estimate is well-known, its proof can be found in, e.g., \cite{DavSza}. The second auxiliary lemma gives a lower bound for the conditional expectation of a vector norm in terms of the unconditional expectation.
\begin{lemma} \label{exp_cond}
Let $S$ be $m \times s$ matrix with i.i.d. standard normal entries and $A$ is $m \times n$ fixed matrix. Let $\EE$ be any event such that $\P(\EE) \ge 1 - e^{-cm}$ for some $c \in (0, 1/2]$.
Then for any fixed $u \in \mathds{S}^{n-1}$ and large enough $m$,
$$ \E( \|A_S u\|_2^2| \EE) \ge \E\|A_S u\|_2^2 - e^{-cm/2}\|A\|^2. $$
If $c = 1/2$, it is enough to take $m \ge 22$ for the statement to hold.
\end{lemma}
\begin{proof}

For any $t > 0$,
\begin{align*}
\P(&\|A_Su\|_2^2 > t|\EE) = 1 - \P(\|A_Su\|_2^2 \le t|\EE) \\
&\ge 1 - \frac{\P(\|A_Su\|_2^2 \le t)}{\P(\EE)} \ge \P(\|A_Su\|_2^2 > t) - e^{-cm},
\end{align*}
since $\P(\EE) \ge 1 - e^{-cm}$. Then,
\begin{align*}
 \E( &\|A_S u\|_2^2| \EE) \ge \int_{t = 0}^{9 m\|A\|^2} \P(\|A_S u\|_2^2 > t| \EE) dt \\
&\ge \int_{t = 0}^{9 m\|A\|^2} \P(\|A_Su\|_2^2 > t) dt - \int_{t = 0}^{9 m\|A\|^2} e^{-cm} dt\\
&\ge \E\|A_Su\|_2^2 - \int_{9m\|A\|^2}^{\infty} \P(\|A_Su\|_2^2 > t) dt -\frac{9 m\|A\|^2}{e^{cm}}.
\end{align*}
To bound the integral term, note that a trivial inequality $\|S^TA\| \le \|S^T\|\|A\|$ implies
$$ \P(\|A_Su\|_2^2 > t) \le \P(\|A_S\|_2^2 > t)\le \P\left(\|S^T\| > \frac{\sqrt{t}}{\|A\|}\right). $$
This allows as to bound
\begin{align}
&\int_{t = 9 m\|A\|^2}^{\infty} \P(\|S^T Au\|_2^2 > t) dt \nonumber\\
&\le \int_{q = 3}^{\infty} \P(\|S^T\| > q \sqrt{m}) 2 q m\|A\|^2 dq \label{v1}\\
&= \|A\|^2 \int_3^{\infty} \exp\left(-\frac{q^2 m}{9\cdot 2}\right) \, 2 q m \, d q \le 18e^{-m/2}\|A\|^2, \label{v2}
\end{align}
using a change of variable $q =\sqrt{t}/\sqrt{m}\|A\|$ in \eqref{v1}, Lemma~\ref{sup_norm} and the fact that $q-2 \ge q/3$ for $ q \ge 3$ in \eqref{v2}.
As a result, 
\begin{align*}
\E( \|A_S u\|_2^2| \EE) &\ge \E\|A_Su\|_2^2 -  e^{-cm}\|A\|^2(9m + 18) \\
 &\ge \E\|A_Su\|_2^2 - e^{-cm/2} \|A\|^2
\end{align*}
for $m$ large enough. Note that $e^{ m/4} \ge 9m + 18$ for all $m \ge 22$. Lemma~\ref{exp_cond} is thus proved.
\end{proof}
\begin{remark}
Note that the reverse inequality between conditional and unconditional expectation (given that we condition on the highly probable event $\EE$ with $\P(\EE) \ge 1 - e^{-cm}$) is straightforward:
$$\E\|A_S u\|_2^2 = \E( \|A_S u\|_2^2| \EE) \P(\EE) + \E( \|A_S u\|_2^2| \EE^c) \P(\EE^c) \ge \E( \|A_S u\|_2^2| \EE)(1 - e^{-cm}).
$$
Thus, for $m$ large enough, we have a tight two-sided estimate:
$$ \E\|A_S u\|_2^2 - e^{-cm/2}\|A\|^2 \le \E( \|A_S u\|_2^2| \EE) \le \frac{\E \|A_S u\|_2^2}{1 - e^{-cm}}.$$
\end{remark}

We are ready to bound the expected improvement after one step of the BGK method.
\begin{proposition}\label{main1} Suppose $A$ is a $m \times n$ matrix with full column rank ($m \ge n$) and let $x_*$ be a solution of the system $Ax = b$.  Let $x_k$ be a fixed vector in $\R^n$. Let $S$ be $m \times s$ matrix with i.i.d. standard normal entries and $x_{k+1}$ be obtained by iteration \eqref{iter}. Then, 
$$
\E_S\|x_{k+1} - x_*\|_2^2 \le \beta \|x_k - x_*\|_2^2,
$$
where
\begin{equation}\label{beta}
\beta =  1 - \frac{s}{10m \kappa^2(A)} - \frac{e^{-m/4}}{10m}.
\end{equation}
\end{proposition}

\begin{proof}
Since
\begin{align*}
x_{k+1} - x_* &= x_k - x_* - (A_S^{\dagger})(A_S x_k - A_S x_*)  \\
&= (\Id - A_S^{\dagger}A_S)(x_k - x_*),
\end{align*}
we have
$$
\E \|x_{k+1}-x_*\|_2^2 = \E \|(\Id - A_S^{\dagger}A_S)(x_k - x_*)\|_2^2.
$$
Since $A_S^{\dagger}A_S$ is an orthogonal projector,
$$
\E \|(\Id - A_S^{\dagger}A_S)u\|_2^2 = \|u\|_2^2 - \E \|A_S^{\dagger}A_Su\|_2^2.
$$
So, our goal is to prove that for any fixed $u \in \mathds{S}^{n-1}$
\begin{equation}\label{main_2}
\E \|A_S^{\dagger}A_Su\|_2^2 \ge 1 - \beta, \text{ where } \beta \text{ is defined by \eqref{beta}}.
\end{equation}
Now, for  any $\gamma > 0$, by the total expectation theorem,
\begin{align}
\E\|A_S^{\dagger}A_Su\|_2^2 \ge \E(\sigma_{min}^2(\Ac)&\cdot \|A_S u\|_2^2) \nonumber\\
\ge \E \big( \sigma_{min}^2(\Ac) \|A_S u\|_2^2 \,&\vert \,\sigma_{min}^2(\Ac) \ge \gamma^{-2}\big) \cdot \P\big(\sigma_{min}^2(\Ac) \ge \gamma^{-2}\big) \nonumber\\
\ge \frac{1}{\gamma^2} \E( \|A_S u\|_2^2| \EE) \P(\EE&),
\end{align}
where $\EE := \{\|A_S\| \le \gamma\}.$ Now, with $\gamma = 3 \sqrt{m}\|A\|$, we have: 

\begin{enumerate}
\item since $\|S^T A\| \le \|S^T\|\|A\|$,
$$
 \P(\EE) = \P (\|S^TA\| \le 3 \sqrt{m}\|A\|) \ge  \P(\|S^T\| \le 3 \sqrt{m})
 $$ 
 and, by Lemma~\ref{sup_norm},
$$
\P (\|S^T\| \le 3 \sqrt{m})  \ge 1 - \exp(-m/2).
$$
\item Then, by Lemma~\ref{exp_cond} applied to the event $\EE$, 
$$ \E( \|A_S u\|_2^2| \EE) \ge \E( \|A_S u\|_2^2) - e^{-m/4}\|A\|^2. $$
\item Finally, unconditional expectation can be computed directly: if $\{S_i\}_{i = 1}^s$ denote the columns of the matrix $S$,
$$
\E\|A_S u\|_2^2 = \E \sum_{i=1}^s \langle S_i^T, Au\rangle^2 = s \|Au\|_2^2 \ge s \sigma_{min}^2(A).
$$
\end{enumerate}
Combining the three estimates above, we obtain
\begin{align}\label{main1:final_est}
\E \|A_S^{\dagger}A_Su\|_2^2 &\ge \frac{(s \sigma_{min}^2(A)-e^{-m/4}\|A\|^2)(1 - e^{-m/2})}{9 m \|A\|^2} \nonumber\\
&\ge \frac{s \sigma_{min}^2(A)}{10m \sigma^2_{max}(A)} - \frac{e^{-m/4}}{10m}, 
\end{align}
for any $m \ge 22$. This concludes the proof of Proposition~\ref{main1}.
\end{proof}
\begin{remark}
The fact that it is enough to prove an estimate \eqref{main_2} directly follows from the result of [\cite{GowRic}, Theorem~]. We decided to include the derivation of this step for the completeness of exposition. 
\end{remark}
{\bfseries Proof of Theorem~\ref{theor1}.} Note that due to the condition number assumption, $\kappa^2(A) \le e^{m/4}/3$, the exponential term in the one step estimate from Proposition~\ref{main1} becomes negligible:
$$
1 - \frac{s}{10m \kappa^2(A)} - \frac{e^{-m/4}}{10m} \le 1 - \frac{s}{15m \kappa^2(A)}.
$$
Thus,
 \begin{align*}
\E\|x_{k} - x_*\|_2^2  &= \E_{S_1} \E_{S_2}  \ldots \E_{S_k} \|x_{k} - x_*\|_2^2  \\
&\le \left[ 1 - \frac{s}{15m \kappa^2(A)}\right]^k \|x_0- x_*\|_2^2.
\end{align*}
Here, $\E_{S_1}, \ldots, \E_{S_k}$ refer to the randomness of choosing a matrix $S_i$ (independent from all Gaussian sketches $S_1, \ldots, S_{i-1}$ that were used during previous steps). The last inequality is guaranteed by Proposition~\ref{main1}. Thus Theorem~\ref{theor1} is proved.

\bigskip

\subsection{Proof of Theorem~\ref{theor2}, convergence estimate via the mix of Frobenius and operator norms}\label{proof2}

The first auxiliary lemma is a direct corollary of a matrix deviation inequality (see, e.g., \cite[Theorem~9.1.1]{VerHDP}). We will use it to make an estimate for the norm $\|S^TA\|$ (more accurate than a trivial estimate $\|S^T A\| \le \|S^T\|\|A\|$ that was used in the proof of Theorem~\ref{theor1}).

\begin{lemma}\label{gordon}
Let $S$ be $m \times s$ matrix with i.i.d. standard normal entries and $A$ is $m \times n$ fixed matrix. Let $\mathds{S}^{n-1}$ denote the unit sphere in $\R^n$. Then the following holds with some absolute constant $C > 0$:
$$ 
\E \sup\limits_{w \in A\mathds{S}^{n-1}}\|S^T w\|_2 \le \sqrt{s} \|A\| + \|A\|_F.
$$
\end{lemma}
\begin{proof}
Chevet's inequality (in the form obtained in \cite[Corollary~2.4]{gordon1985some}, see also \cite[Exercise~8.7.4]{VerHDP}) states that for any $U \subset \R^n$
\begin{equation}\label{chevet}
\E \sup\limits_{x \in U} \|S^T x\|_2 = \E \sup \limits_{x \in U, y \in \mathds{S}^{s-1}}  \langle S^T x, y\rangle \le \omega(U) \rad(\mathds{S}^{s-1}) + \omega( \mathds{S}^{s-1}) \rad(U),
\end{equation}
where radius of the set $\rad(U) :=\frac{1}{2} \sup\{ \|x-y\|_2\, : \, x, y \in U\}$, and Gaussian width $\omega(U)$ is defined by 
$$
\omega(U) := \E_g \sup\limits_{x \in U}\langle g,x \rangle \quad \text{ where } \quad g \sim N(0, I_n).
$$

Now, $\rad(\mathds{S}^{s-1}) = 1$ and $\omega(\mathds{S}^{s-1}) = \E \|g\|_2 \le \sqrt{s}$ (for $g \sim N(0, I_s)$) by Jensen's inequality. Moreover, in the case when $U$ is an ellipsoid $U = A\mathds{S}^{n-1}$, Gaussian width as well as the $L_2$-norm bound for the element in $U$ are bounded in terms of the norms of the matrix $A$, namely, $\omega(U) \le \|A\|_F$ (see, e.g., \cite[Section 7.6]{VerHDP}) and 
$$\rad(U) \le \frac{1}{2} \sup_{x, y \in U} (\|x\|_2 +\|y\|_2) \le \sup_{x \in U}\|x\|_2 = \sup_{y \in \mathds{S}^{n-1}}\|Ay\|_2 = \|A\|.$$ 
Substituting radii and Gaussian width values to the right hand side of \eqref{chevet} concludes Lemma~\ref{gordon}.
\end{proof}

The second auxiliary lemma estimates the norm $\|S^TAx\|_2$ with high probability. It relies on the following version of Cram\'er's concentration inequality (see, e.g., \cite{BouLugMas})
 \begin{theorem}[Cram\'er's Theorem]\label{cramer}
 Let $X$ be a random variable, such that for all $\lambda \in \R$ its moment generating function is finite: $\E e^{\lambda X} < +\infty$.
 Let $X_1, \ldots X_n$ be i.i.d. copies of $X$, and set $S = \sum_{i = 1}^n X_i$.  Then for any $a < \E X$ we have 
 $$\P (S/n < \alpha) \le \exp(-I(\alpha)\cdot n),$$
  where the function $I: \R \to [0, +\infty]$ is defined by 
 \begin{equation}\label{i-function}
   I(\alpha) = \sup\limits_{t \in \R} (t\alpha - \log \E \exp(tX)).  
   \end{equation}
 \end{theorem}

\begin{lemma}\label{small_ball}
Let $S$ be $m \times s$ matrix with i.i.d. standard normal entries and $v \in \R^m$ is a fixed vector. Then
$$ \P\big( \|S^T v\|_2^2 > \|v\|^2 s/10\big) \ge 0.5. $$
\end{lemma}
\begin{proof}
Note that a random variable $Z := \|S^Tv\|_2^2/\|v\|_2^2$ has a distribution of a sum of the squares of $s$ independent standard normal Gaussian random variables
$$Z = \sum_{i =1}^{s} \left(\sum_{j = 1}^m S_{ij}^T \frac{v_j}{\|v\|_2^2}\right)^2 \sim \sum_{i =1}^{s} Z_i^2.$$

So, for any $i = 1, \ldots, s$ a random variable $Z_i^2$ has a chi-squared distribution with one degree of freedom. A direct computation involving the moment generating function of $\chi^2$ shows that for any $\alpha < 1$, the function $I(\alpha)$ defined by~\eqref{i-function} is
$$
I(\alpha) = \frac{\alpha - 1 + \ln(\alpha^{-1})}{2}.
$$
Therefore, Cram\'er's Theorem~\ref{cramer} with $s \ge 1$ gives 
$$
 \P\left( \frac{\|S^T v\|_2^2}{\|v\|^2} \le \frac{s}{10}\right) \le \exp\left(- \frac{s (\ln(10) - 0.9)}{2}\right) \le \frac{1}{2}.
 $$
\end{proof}

We next turn to bounding the error in one iteration by the error in the previous iteration.

\begin{proposition}\label{main2} Suppose $A$ is a $m \times n$ matrix with full column rank ($m \ge n$) and let $x_*$ be a solution of the system $Ax = b$.  Let $x_k$ be a fixed vector in $\R^n$. Let $S$ be a $m \times s$ matrix with i.i.d. standard normal entries and $x_{k+1}$ is obtained by iteration \eqref{iter}. Then, 

$$
\E_S\|x_{k+1} - x_*\|_2^2 \le \beta \|x_k - x_*\|_2^2,
$$
where
\begin{equation}\label{beta2}
\beta =  1 - \frac{s\sigma_{min}^2(A)}{80(\sqrt{s} \|A\| + \|A\|_F)^2}
\end{equation}
\end{proposition}
\begin{proof}
As 
$$
\E \|(\Id - A_S^{\dagger}A_S)u\|_2^2 = \|u\|_2^2 - \E \|A_S^{\dagger}A_Su\|_2^2
$$
(see \eqref{main_2} and above), it is enough to show that for any fixed $u \in \mathds{S}^{n-1}$
\begin{equation}\label{eq:main2}
\E \|A_S^{\dagger}A_Su\|_2^2 \ge \frac{s \sigma^2_{min}(A)}{80(\sqrt{s} \|A\| + \|A\|_F)^2}.
\end{equation}
We apply the total expectation theorem (just like in the proof of Proposition~\ref{main1}, but this time we are conditioning on the norm of $\|A_Su\|_2$). For any parameter $\gamma^2 > 0$,
\begin{align}\label{eq:main2_2}
\E\|A_S^{\dagger}A_Su\|_2^2 &= \E\left(\big\|A_S^{\dagger}\frac{A_Su}{\|A_Su\|^2_2}\big\|_2^2 \cdot\|A_Su\|_2\right) \nonumber\\
&\ge \E\left(\big\|A_S^{\dagger}\frac{A_Su}{\|A_Su\|_2}\big\|_2^2 \|A_Su\|^2_2 \,\biggm\vert \, \|A_Su\|^2_2\ge \gamma^2\right) \cdot \P\big(\|A_Su\|^2_2 \ge \gamma^2\big) \nonumber\\
&\ge \gamma^2 \E(\inf_{v \in \mathds{S}^{n-1}}\|A_S^{\dagger}v\|_2^2 \, \big| \, \EE_{\gamma}) \P(\EE_{\gamma})
\end{align}
where the event $\EE_{\gamma} := \{\|A_Su\|^2_2 \ge \gamma^2\}$. Furthermore,
\begin{align*}
\E\big(\inf_{v \in \mathds{S}^{n-1}}\|A_S^{\dagger}v\|_2^2 \, \big| \, \EE_{\gamma}\big) &=  \E\left((\sup\limits_{v \in \mathds{S}^{n-1}}\|A_Sv\|_2^2)^{-1} \, \biggm\vert \, \EE_{\gamma}\right) \\
&= \E\left((\sup\limits_{v \in \mathds{S}^{n-1}}\|A_Sv\|_2)^{-2} \, \biggm\vert \, \EE_{\gamma}\right) \\
&\ge \E^{-2}\left(\sup\limits_{v \in \mathds{S}^{n-1}}\|A_Sv\|_2 \, \biggm\vert \, \EE_{\gamma}\right),
\end{align*}
since $f(x) = x^2$ is a monotone function on $x \ge 0$ (and so $\sup( \|.\|^2) = \sup^2 \|.\|$) and by Jensen's inequality applied to a convex function $g(x) = x^{-2}$. To estimate the denominator from above, we can use the total probability theorem again, namely, for any event $\EE_{\gamma}$
$$
\E (\sup\limits_{v \in \mathds{S}^{n-1}}\|A_Sv\|_2 \, \big| \, \EE_{\gamma}) \le \frac{\E \sup_{v \in \mathds{S}^{n-1}}\|A_Sv\|_2 }{\P(\EE_{\gamma})}.
$$
Finally, $\E\sup\limits_{v \in \mathds{S}^{n-1}}\|A_Sv\|_2$ can be estimated by Lemma~\ref{gordon} as
$$ 
\E \sup\limits_{w \in A\mathds{S}^{n-1}}\|S^T w\|_2 \le \sqrt{s} \|A\| + C \|A\|_F.
$$

Combining the last two estimates with \eqref{eq:main2_2}, we obtain
$$
\E\|A_S^{\dagger}A_Su\|_2^2 \ge \frac{ \gamma^2\P^3(\EE_{\gamma})}{ (\sqrt{s} \|A\| + \|A\|_F)^2}.
$$
The numerator can be estimated by the Lemma~\ref{small_ball} if we take $v = Au$ and $\gamma^2 =  s\|Au\|_2^2/10$:
\begin{align*}
\gamma^2 \cdot \P^3(\|A_Su\|^2_2 \ge \gamma^2) &\ge \frac{s \|Au\|_2^2}{10} \cdot \P^3(s\|A_Su\|^2\ge \frac{s\|Au\|_2^2}{10}) \\
&\ge\frac{\|Au\|_2^2 s}{10}\cdot\frac{1}{2^3} \ge \frac{s \sigma_{min}^2(A)}{80}.
\end{align*}
So,
$$
\E\|A_S^{\dagger}A_Su\|_2^2 \ge \frac{ \gamma^2\P^3(\|A_Su\|^2_2 \ge \gamma^2)}{ L^2_{\|A\|, \|A\|_F}} \ge \frac{\sigma_{min}^2(A) s}{80L^2_{\|A\|, \|A\|_F}},
$$
where $L_{\|A\|, \|A\|_F} = \sqrt{s} \|A\| + \|A\|_F$. This concludes the proof of Proposition~\ref{main2}.

\end{proof}

{\bfseries Proof of Theorem~\ref{theor2}.} We have
\begin{align*}
\E\|x_{k} - &x_*\|_2^2  = \E_{S_1} \E_{S_2}  \ldots \E_{S_k} \|x_{k} - x_*\|_2^2 \\
&\le \left[ 1 - \frac{s\sigma_{min}^2(A)}{80(\sqrt{s} \|A\| + \|A\|_F)^2}\right]^k \|x_0- x_*\|_2^2.
\end{align*}
Here, $\E_{S_1}, \ldots, \E_{S_k}$ refer to the randomness of choosing a matrix $S_i$ (independent from all Gaussian sketches $S_1, \ldots, S_{i-1}$ that were used during previous steps). The last inequality is guaranteed by Proposition~\ref{main2}. This concludes the proof of Theorem~\ref{theor2}.

\bigskip

\subsection{Proof of Theorem~\ref{theor3}, finite block Gaussian case}\label{proof3}

The following lemma is a direct corollary of the standard Bernstein's inequality for sub-exponential random variables (see e.g., \cite[Corollary 2.8.3]{VerHDP}), specialized for the case of standard normal random variables:
\begin{lemma}\label{sub_exp_ber}  Let $X_1, \ldots, X_n$ be independent $N(0,1)$ random variables. 

\begin{enumerate}[a)]

\item Then, for every $t \ge 0$, we have
$$
\P\left(\big|\sum_{i = 1}^n (X_i^2 - 1) \big| \ge t \right) \le 2 e^{- \min\{t^2/(8n), \, t/6 \}}.
$$
\item Let $Y_1, \ldots, Y_n$ be independent $N(0,1)$ random variables, also independent with all $X_i$, $i = 1, \ldots n$. Then, for every $t \ge 0$, we have
$$
\P\left(\big|\sum_{i = 1}^n X_i Y_i \big| \ge t \right) \le 2 e^{- \min\{t^2/(4n), \, 2t/5 \}}.
$$
\end{enumerate}
\end{lemma}

We will utilize the following definition that characterizes when a set of sketching matrices is suitable for our needs.

\begin{definition}[Good collection]\label{good_set}
We will call a set $\mathcal{S} = \{S^{(1)}, \ldots, S^{(N)}\}$ of $m \times s$ real matrices ``good", if the following conditions hold:
\begin{enumerate}
\item all $S^{(k)} \in \mathcal{S}$ have bounded operator norm: 
$$\|S^{(k)}\|\le 3\sqrt{m};$$
\item for all pairs $(j,i) \ne (r, i) \in [m] \times [s]$: 
$$\left|\sum_{k = 1}^N S_{ji}^{(k)} S^{(k)}_{ri}\right| \le N/4m$$ 
(so, all the entries of a sampled matrix from the collection $\mathcal{S}$ are empirically approximately independent);
\item for any $(j,i) \in [m] \times [s]$
$$\left|\sum_{k=1}^N (S_{ji}^{(k)})^2 - N \right| \le \frac{N}{2}$$
(so, all the entries of a sampled matrix from the collection $\mathcal{S}$ have not too small an empirical second moment).
\end{enumerate}
Here, $S_{ij}^{(k)}$ denotes the $(i,j)$-entry of the matrix $S^{(k)}$ and $N = |\mathcal{S}|$.
\end{definition}
Note that the conditions (2) and (3) from the Definition~\ref{good_set} imply that, if in the process of sampling entries of the matrices in the collection $\mathcal{S}$ (uniformly with replacement), the sample covariance matrix obtained would be reasonably close to the identity. 

Now we will check that a random collection of standard Gaussian matrices is likely a good collection.

\begin{proposition}\label{random_set_is_good}
 Let $S^{(1)}, \ldots, S^{(N)}$ be independent $m \times s$ real random matrices with i.i.d. standard normal entries, and $c > 3$ an arbitrary constant. For all cardinalities $N$ with $64cm^2\ln m \leq N \leq e^{m/3}$, with high probability $1 - 1.1m^{3-c}$, the collection $\mathcal{S} = \{S^{(1)}, \ldots, S^{(N)}\}$ is ``good" (in the sense of Definition~\ref{good_set}).
\end{proposition}

\begin{proof}
Let us compute the probability that a random set of $N$ standard Gaussian matrices is not good, namely, at least one of the conditions of Definition~\ref{good_set} is violated.

By Lemma~\ref{sup_norm}, combined with the union bound,
$$
\P\big(\exists S \in \mathcal{S}: \, \|S\| > 3 \sqrt{m}\big) \le N\exp(-m/2).
$$
For any fixed indices $i \in [s]$, $j \in [m]$, and $r \in [m]$ such that $r \ne j,$ all the entries $S_{ji}^{(k)}$ and $S_{ri}^{(k)}$ for $k = 1, \ldots, N$ are mutually independent Gaussian random variables. Hence, by Lemma~\ref{sub_exp_ber} with $t = N/4m$,
$$
\P\left(\left| \sum_{k=1}^N S_{ji}^{(k)} S_{ri}^{(k)} \right| \ge \frac{N}{4m}\right) \le 2\exp(-N/64m^2).
$$
Taking a union bound over all pairs of indices $(j,i)$ and $(r,i)$, we get that (2) does not hold for for the collection $\mathcal{S}$ with probability at most $2m^2 s\exp(-N/64m^2)$. 

By Lemma~\ref{sub_exp_ber} with $t = N/2$,
\begin{align*}
\P\left(\sum_{k=1}^N (S_{ji}^{(k)})^2 \le \frac{N}{2}\right) \le \P\left(\left|\sum_{k=1}^N (S_{ji}^{(k)})^2 - N \right| \ge \frac{N}{2}\right) \le 2\exp(-N/32).
\end{align*}
Taking a union bound over all pairs of indices $(j, i) \in [m] \times [s]$, the probability that (3) does not hold for for $S$ is bounded by $2ms\exp(-N/32)$. Therefore, combining all three probabilities of the exceptional events, if $64cm^2 \ln m \le N \le \exp(m/3)$, for $c > 3$, the probability that a random collection $\mathcal{S}$ of cardinality $N$ is ``good" is at least
\begin{align*} 1 - 2m^2 s e^{-N/64m^2} - 2mse^{-N/32}- Ne^{-m/2} &\ge 1 - m^{3-c} - e^{-m} - e^{-m/6} \\
&\ge 1 - 1.1m^{3 - c}.
\end{align*}
This concludes the proof of Proposition~\ref{random_set_is_good}.
\end{proof}

Now we will prove that, given a good collection $\mathcal{S}$, the iterative process \eqref{iter} (choosing a sketch $S$ from $\mathcal{S}$ randomly at each iteration) enjoys the same convergence bound as the analogous process sampling a new Gaussian matrix at every step.
\begin{proposition}\label{main3} Suppose $A$ is a $m \times n$ matrix with full column rank ($m \ge n$) and let $x_*$ be a solution of the system $Ax = b$.  Let $x_k$ be a fixed vector in $\R^n$. Let $\mathcal{S} = \{S^{(1)}, \ldots, S^{(N)}\}$ be a ``good" set of $m \times s$ matrices (in the sense of Definition~\ref{good_set}). At every iteration we choose a random matrix $S$ uniformly at random from $\mathcal{S}$ (with replacement), and iterate according to \eqref{iter}. Then, 
$$
\E\|x_{k+1} - x_*\|_2^2 \le \left[ 1 - \frac{s}{36m \kappa^2(A)}\right] \|x_{k} - x_*\|_2^2.
$$
\end{proposition}
\begin{proof}
As 
$$
\E \|(\Id - A_S^{\dagger}A_S)u\|_2^2 = \|u\|_2^2 - \E \|A_S^{\dagger}A_Su\|_2^2
$$
(see \eqref{main_2} and above), it is enough to show that for any fixed $u \in \mathds{S}^{n-1}$
\begin{equation}\label{eq:main3}
\E \|A_S^{\dagger}A_Su\|_2^2 \ge \frac{s}{36 m \kappa^2(A)},
\end{equation}
where expectation is taken over the random choices of $S \in \mathcal{S}$ and $A_S = S^TA$. Since $\mathcal{S}$ satisfies property (1) of Definition~\ref{good_set}, for any $S \in \mathcal{S}$ we have 
$$\sigma_{min}^2(\Ac) = \frac{1}{\sigma_{max}^2(S^TA)} \ge \frac{1}{\|S^T\|^2\|A\|^2} \ge \frac{1}{9m\|A\|^2}.$$ 
Thus,
\begin{align}\label{eq:main3-1}
\E\|A_S^{\dagger}A_Su\|_2^2 &\ge \E(\sigma_{min}^2(\Ac)\cdot \|A_S u\|_2^2) \nonumber\\
&\ge \frac{1}{9m\|A\|^2} \E( \|A_S u\|_2^2).
\end{align}
Now we are going to estimate the expectation term. Let $v := Au$, $S_{ji}$ denote the $(j, i)$-element of a random matrix $S$, and $S^{(k)}_{ji}$ denotes the $(j, i)$-element of a fixed matrix $S^{(k)}$ (from the collection). Then
\begin{align}
\E \|A_S u\|_2^2 &= \E \sum_{i=1}^s \big(\sum_{j = 1}^n S_{ji} v_j\big)^2 \nonumber\\   
 &=  \sum_{i=1}^s \left[ \sum_{j = 1}^m \E (S_{ji}^2) v_j^2 + \sum_{j \ne r; j,r = 1}^m \E(S_{ji} S_{ri}) v_j v_r\right] \nonumber\\   
 &=  \sum_{i=1}^s \left[ \sum_{j = 1}^m \frac{1}{N} \sum_{k = 1}^N (S^{(k)}_{ji})^2 v_j^2 + \sum_{j \ne r; j,r = 1}^m\frac{1}{N} \sum_{k = 1}^N(S^{(k)}_{ji}S^{(k)}_{ri}) v_j v_r\right] \nonumber\\   
&\ge \sum_{i=1}^s \left[ \sum_{j = 1}^m  \frac{1}{N}\sum_{k = 1}^N (S^{(k)}_{ji})^2 v_j^2\right] - \sum_{i=1}^s \left[\sum_{j \ne r; j,r = 1}^m \frac{1}{N} \big| \sum_{k = 1}^N S^{(k)}_{ji}S^{(k)}_{ri} \big| \cdot |v_j v_r| \right]  \label{1}\\   
 &\ge \sum_{i=1}^s \left[ \sum_{j = 1}^m  \frac{1}{2} v_j^2\right] - \sum_{i=1}^s \left[\sum_{j \ne r; j,r = 1}^m \frac{1}{4m}|v_j v_r|\right]  \ge \frac{s}{4}\|v\|^2. \label{2}
\end{align}
In \eqref{1}, we used the properties (2) and (3) of Definition~\ref{good_set}, and the last line holds since
$$
\sum_{j \ne r} |v_j v_r| \le 0.5\sum_{j \ne r} (v_j^2 + v_r^2) \le m \|v\|_2^2.
$$
Now recall that $\|v\|^2 = \|Au\|^2 \ge \sigma_{min}^2 (A)$. Combining \eqref{eq:main3-1} with \eqref{2} we conclude that
$$
\E\|A_S^{\dagger}A_Su\|_2^2 \ge \frac{1}{9m\|A\|^2} \frac{s}{4}\sigma_{min}^2 (A),
$$
which completes the proof.
\end{proof}

{\bfseries Proof of Theorem~\ref{theor3}.} Given the constraints on the size of the collection $N$, a random collection $\mathcal{S}$ of standard normal matrices will be a ``good" set (in the sense of Definition~\ref{good_set}) with probability at least $1 - 1.1m^{3-c}$ (by Proposition~\ref{random_set_is_good}). Conditioned on this high probability event, the iteration process will converge exponentially fast as given by the statement of Theorem~\ref{theor3}:
\begin{align*}
\E&(\|x_{k} - x_*\|_2^2| \, \mathcal{S} \text{ is good }) \\
&= \E_{S_1} \E_{S_2}  \ldots \E_{S_k} (\|x_{k} - x_*\|_2^2| \, \mathcal{S} \text{ is good }) \\
&\le \left[ 1 - \frac{s}{36m \kappa^2(A)}\right]^k \|x_{0} - x_*\|_2^2.
\end{align*}
Here, $\E_{S_1}, \ldots, \E_{S_k}$ refer to the randomness of choosing a matrix $S_i \in \mathcal{S}$ (due to the sampling with replacement, random matrices $S_1, S_2, \ldots$ are independent). The last inequality is guaranteed by Proposition~\ref{main3}. This concludes the proof of Theorem~\ref{theor3}.

\subsection{Proof of Theorems~\ref{theor4} for the inconsistent systems}\label{proof_with_error}
We again begin with an auxiliary lemma.
\begin{lemma}[Independence lemma]\label{independence_lemma}
Suppose $A$ is a $m \times n$ matrix with full column rank ($m \ge n$) and let $x_*$ be a solution of the least squares minimization problem~\eqref{minimization_problem}, and $e := Ax_* - b$. Let $S$ be an $m \times s$ random matrix with i.i.d. standard normal entries. Then two random variables $\|A_S^{\dagger}\|_2$ and $\|S^Te\|_2$ are independent.
\end{lemma}

\begin{proof}
Note that since $x_*$ is a minimizer for \eqref{minimization_problem}, we can assume without loss of generality that the error term $e$ is orthogonal to the image $\Im(A) := \{Ax, x \in \R^n\}$. Indeed, if $e = e_1 + Au$, where $\langle e_1, Au \rangle = 0$, then $\|Ax_{**} - s\|_2 \le \|A x_* - s\|$ for $x_{**} = x_* - u$. 

The orthogonality between the error vector $e$ and the image of $A$ implies that random variables $\|A_S^{\dagger}\|_2$ and $\|S^Te\|_2$ are independent. Indeed, it is straightforward to check that 
$$\|(S^TA)^{\dagger}\|_2 = f(Q_{[1:n]}) \quad \text{ and } \quad \|S^Te\|_2 = g(Q_{[n+1:m]}).$$ 
Here, $Q_{[a:b]}$ denotes the restriction on the range of columns $a, \ldots, b$ of the matrix $Q$, and $Q \in \R^{s \times m}$ is a standard Gaussian random matrix defined by $Q := S^TU$ (orthogonal matrix $U$ is defined by the singular value decomposition of $A = U \Sigma_A V^T$). Lemma~\ref{independence_lemma} is proved.
\end{proof}

\begin{proposition}\label{error_one_step_estimate_fin} Suppose $A$ is a $m \times n$ matrix with full column rank ($m \ge n$) and let $x_*$ be a solution of the least squares minimization problem~\eqref{minimization_problem}, and $e := Ax_* - b$.  Let $\mathcal{S} = \{S^{(1)}, \ldots, S^{(N)}\}$ be a ``good" set of $m \times s$ matrices (in the sense of Definition~\ref{good_set}) and additionally for all $S^{(k)} \in \mathcal{S}$
\begin{equation}\label{add_condition}
\sigma_{min}((U^TS^{(k)})_{[1:n,:]}) \ge (\sqrt{n} - \sqrt{s})/2 \text{ for } U: A = U\Sigma V^T \text{ is the SVD decomposition of matrix } A,
\end{equation}
then
$$
\E\|A_S^\dagger S^Te\|_2^2 \le \frac{8s\|e\|^2_2}{(\sqrt{n} - \sqrt{s})^2 \sigma_{min}^2(A)},
$$
where expectation is taken over a uniform draw of a matrix $S$ from the collection $\mathcal{S}$.\end{proposition}
\begin{proof} Since every matrix in $\mathcal{S}$ satisfies the conditions of the Independence Lemma~\ref{independence_lemma},
$$
\E\|A_S^\dagger S^Te\|_2^2 \le \E\|A_S^\dagger\|_2^2\cdot\E\| S^Te\|_2^2 \quad \text{ for any } S \in \mathcal{S}.
$$
Then, $\E\|S^Te\|_2^2 \le 2s\|e\|^2$ by the conditions (2) and (3) from the Definition~\ref{good_set} ensure that   (with the proof along the lines of the lower estimate for $\E \|A_S u\|_2^2$ in Proposition~\ref{main3}, changing signs in \eqref{1} and \eqref{2} and using the other side of the condition (3)). 

For the first multiple, let us note that
\begin{align*}
 \E\|A_S^\dagger\|_2^2 &= \E \sigma_{min}^{-2}(A^TS) = \E  \left((\inf\limits_{x \in \mathds{S}^{n-1}} \|\Sigma U^T S x\|_2^2)^{-1}\right) \\
 &\le \sigma_{min}^{-2}(A) \E\sigma_{min}^{-2}(U^TS) \le \frac{4}{(\sqrt{n} -  \sqrt{s})^2\sigma_{min}^{2}(A)}
\end{align*}
due to the additional condition~\eqref{add_condition}. This concludes the proof of Proposition~\ref{error_one_step_estimate_fin}.
\end{proof}

\bigskip

{\bfseries Proof of Theorem~\ref{theor4}.}
Let $U$ be a deterministic orthonormal matrix from the SVD decomposition of the matrix $A$ (namely, $A = U \Sigma V^T$). Then $(U^TS^{(k)})_{[1:n,:]}$ are i.i.d. standard normal matrices independent of each other (for $k = 1, \ldots, N$, where $S^{(k)} \in \mathcal{S}$). By Proposition~\ref{random_set_is_good} and the fact that 
\begin{equation}\label{smin_gaus}
\sqrt{n} - \sqrt{s} - t \le \sigma_{\min}(S) \quad \text{ with probability at least } 1 - \exp(-t^2/2)
\end{equation}
(see, e.g., \cite{DavSza}) for a standard normal $S$, a randomly drawn collection of Gaussian $m \times s$ matrices satisfy the extended definition of the ``good" collection (Definition~\ref{good_set} and condition~\eqref{add_condition}) with probability $1 - 1.1m^{3-c}  - \exp(-(\sqrt{n} - \sqrt{s})^2/16)$ if  $64cm^2\ln m \leq N \leq \exp((\sqrt{n} - \sqrt{s})^2/16)$.

Since $b = Ax_* - e$, we have the following relation after one iteration of \eqref{iter}:
$$
x_{k+1} - x_* = (\Id - A_S^{\dagger}A_S)(x_k - x_*) - A_S^{\dagger}S^Te,
$$
where $S$ is drawn from the collection $\mathcal{S}$ at random. Then, by the triangle inequality,
\begin{equation}\label{noisy_1}
  \E\|x_{k+1} - x_*\|_2^2  \le \E \|(I - A_S^{\dagger}A_S)(x_k - x_*)\|_2^2 + \E\|A_S^\dagger S^Te\|_2^2.
\end{equation}
Given the constraints on the size of the collection $N$ and the size of the sample, a random collection $\mathcal{S}$ of standard normal matrices will be a ``good" set and additionally satisfy the condition~\eqref{add_condition} with probability at least $1 - 1.1m^{3-c}$. Conditioned on this high probability event, the first term of \eqref{noisy_1} is bounded by Proposition~\ref{main3}, the second term is bounded by Lemma~\ref{error_one_step_estimate_fin}:
$$
\E_{S_1, \ldots, S_k}\|x_{k+1} - x_*\|_2^2 \le \beta \E_{S_1, \ldots, S_{k-1}} \|x_{k} - x_*\|_2^2 + \frac{8s\|e\|^2_2}{(\sqrt{n} - \sqrt{s})^2 \sigma_{min}^2(A)},
$$
where 
$$\beta = 1 - \frac{s}{36 m \kappa^2(A)}.$$
Recursively estimating $\E_{S_1, \ldots, S_{i}} \|x_{i} - x_*\|_2^2$ for $i = (k-1), \ldots, 1$, we get to 
\begin{align*}
\E_{S_1, \ldots, S_k}\|x_{k+1} - x_*\|_2^2 &\le \beta \|x_{0} - x_*\|_2^2 + \sum_{i = 0}^{k-1} \beta^k \frac{8s\|e\|^2_2}{(\sqrt{n} - \sqrt{s})^2 \sigma_{min}^2(A)} \\
&\le \beta \|x_{0} - x_*\|_2^2 + \frac{1}{1 - \beta} \frac{8s\|e\|^2_2}{(\sqrt{n} - \sqrt{s})^2 \sigma_{min}^2(A)}.
\end{align*}
Theorem~\ref{theor4} is thus proved.

\bigskip

To prove Theorem~\ref{theor5} we need the following fact about Gaussian random matrices:
\begin{lemma}\label{gaussian_inv}
If $X$ is an $n \times s$ random matrix with i.i.d. $N(0,1)$ entries, such that $s < \alpha n$ for some constant $\alpha < 1$, and $n$ is large enough, then
$$
\E\sigma_{min}^{-2}(X) \le \frac{20}{(\sqrt{n} - \sqrt{s})^2}.
$$
\end{lemma}
\begin{proof}
By definition of the expectation,
\begin{align}\label{s_min_eq}
\E\sigma_{min}^{-2}(X) &= \int_0^{\infty} \P(\sigma_{min}^{-2}(X)\ge t) dt = \int_0^{\infty} \P(\sigma^2_{min}(X) < 1/t) dt \nonumber\\
&\le a + \int_{a}^{\infty} \P(\sigma^2_{min}(X) < 1/t) dt
\end{align}
for any $a > 0$. We take $a = e^2/(\sqrt{n} - \sqrt{s})^2$. Now, with the change of variable $r = t^{-1}n^{-1}$ we can rewrite the second term as
$$
\int_{a}^{\infty} \P(\sigma^2_{min}(X) < 1/t) dt = \int_{0}^{1/an} \frac{\P(\sigma^2_{min}(X) < rn)}{r^2n} dr.
$$
To bound the probability term, we can now use that for any $r > 0$
$$
\P(\sigma^2_{min}(X) < rn) < \Gamma(n-s+2)^{-1} (\sqrt{r}n)^{n-s+1},
$$
which can be obtained as a direct computation using probability density function of the eigenvalues of a Wishart matrix (see, e.g., \cite[Lemma 4.1]{chen2005condition}). Hence, the second term in~\eqref{s_min_eq} can be upper bounded as
\begin{equation}\label{int}
\frac{n^{n-s+1}}{n\Gamma(n-s+2)}\int_0^{1/an} r^{\frac{n-s+1}{2} - 2} dr = \frac{n^{n-s}}{\Gamma(n-s+2)} \cdot \frac{r^{\frac{n-s-1}{2}}}{\frac{n-s-1}{2}} \big|^{1/{an}}_0
\end{equation}
Rewriting the last expression using Stirling's formula and the chosen value of $a$, we get the following estimate:
\begin{align*} 
\eqref{s_min_eq} &\le  \frac{e^2}{(\sqrt{n} - \sqrt{s})^2} + c C^{n-s}\frac{1}{(1 -\sqrt{s/n})(n - s - 1)(n-s)} \\
&\le \frac{e^2}{(\sqrt{n} - \sqrt{s})^2} \left[1 + \frac{\sqrt{2/\pi} C^{n-s}}{n (1 -\sqrt{s/n})(1 - s/n - 1/n)}\right] \le \frac{2e^2}{(\sqrt{n} - \sqrt{s})^2},
\end{align*}
where $c := \sqrt{\frac{2}{\pi}}e^2$ and $C := \frac{n(\sqrt{n} - \sqrt{s})}{(n -s + 1) \sqrt{n}}$.  In the second step we used that $(\sqrt{n} - \sqrt{s})^2 \le n - s$. Note that $C < 1$ for $n > s > 0$, and $s/n < \alpha$. Hence, the factor in parentheses is less than $2$ for $n$ large enough. Lemma~\ref{gaussian_inv} is proved.
\end{proof}
\begin{remark}
Note that the condition $s < n$ is necessary. As it was proved. in, e.g., \cite[Theorem 4.1]{edelman1989eigenvalues}, $\E\sigma_{min}^{-2}(X)  = \infty$ for the square matrices. However, for any $s \le n - 3$  the integral~\eqref{int} is convergent, so, some finite estimate in terms of $n$ and $s$ could be obtained. The dependence is expected to be much worse since the smallest singular value of an almost square random Gaussian matrix behaves as $n^{-1/2}$ (\cite{RV-rectangular,Feng}) instead of $n^{1/2}$ for a tall matrix (when $s \ll n$). 

For the case when $s/n \to const$ as $n \to \infty$ we obtain the best possible order of the expectation $\E\sigma_{min}^{-2}(X)$ (see also \cite[Proposition 5.1]{edelman1989eigenvalues}). The condition ``$n$ is large enough" is used only in the last estimate of Lemma~\ref{gaussian_inv}. For example, it is enough to have $n \ge 2\sqrt{2/\pi}/(1 - \sqrt{\alpha})(1 - \alpha)$ and $1/n \le (1 -\alpha)/2$.
\end{remark}

The next proposition is an analogue of Proposition~\ref{error_one_step_estimate_fin} for the case when $S$ is sampled from the Gaussian distribution on $\R^{m \times s}$ rather than from a finite collection:
\begin{proposition}\label{error_one_step_estimate} Suppose $A$ is a $m \times n$ matrix with full column rank ($m \ge n$) and let $x_*$ be a solution of the least squares minimization problem~\eqref{minimization_problem}, and $e := Ax_* - b$. Let $S$ be an $m \times s$ random matrix with i.i.d. standard normal entries. Then
$$
\E\|A_S^\dagger S^Te\|_2^2 \le \frac{20s\|e\|^2_2}{(\sqrt{n} - \sqrt{s})^2 \sigma_{min}^2(A)}.
$$
\end{proposition}
\begin{proof} By the Independence Lemma~\ref{independence_lemma},
$$
\E\|A_S^\dagger S^Te\|_2^2 \le \E\|A_S^\dagger\|_2^2\cdot\E\| S^Te\|_2^2 =\E\|A_S^\dagger\|_2^2\cdot s\|e\|^2.
$$
Just like we estimated in Proposition~\ref{error_one_step_estimate_fin}
\begin{align*}
 \E\|A_S^\dagger\|_2^2 \le \sigma_{min}^{-2}(A) \E\sigma_{min}^{-2}(U^TS) \le \frac{20\sigma_{min}^{-2}(A)}{(\sqrt{n} -  \sqrt{s})^2}.
\end{align*}
Here, $U$ is an orthonormal matrix from the SVD decomposition of the matrix $A$ (i.e., $A = U \Sigma V^T$). Since orthogonal transformation of Gaussian matrix is another Gaussian matrix of the size $n \times s$ (as $A$ has full column rank), the last estimate is made by Lemma~\ref{gaussian_inv}.
This concludes the proof of Proposition~\ref{error_one_step_estimate}.
\end{proof}

{\bfseries Proof of Theorem~\ref{theor5}} follows along the lines of the proof of Theorem~\ref{theor4}, starting with~\eqref{noisy_1}, using Proposition~\ref{main2}  instead of Proposition~\ref{main3} and Proposition~\ref{error_one_step_estimate} instead of Proposition~\ref{error_one_step_estimate_fin}.

\section{Numerical experiments}\label{experiments} 
   
In this section, we present numerical experiments to complement our theoretical estimates of the performance of the Gaussian sketch and project methods. Everything was coded and run in MATLAB R2018b, on a 1.6GHz dual-core Intel Core i5, 8 GB 2133 MHz.

We consider two main models of matrices $A$: the first one is an incoherent Gaussian matrix with i.i.d. $N(0,1)$ elements (``Gaussian model"), the second one models a coherent matrix with almost co-linear rows, $A_{ij}\sim Unif[0.8,1]$ (``coherent model"). Unless otherwise stated, we consider matrices of size $m = 50000$ by $n = 500$. As mentioned earlier, all of the Kaczmarz methods can be described in terms of iteration \eqref{iter}. We run the iteration process until the fastest method reaches a relative error threshold (1e-4, unless otherwise stated), or maximal wait time, or specified maximal number of iterations. Relative error is defined as $\|x_k - x_*\|_2^2/\|x_*\|_2^2$. Time is always measured in seconds (cputime). We generate the solution of a system $x_*$ as a standard normal random vector (and define $b = Ax_*$), so we do not need to worry about the case when $\|x_*\|_2 = 0$. We use $x_0 = 0$ as an initial point.

The standard Kaczmarz method \eqref{standardRK} can be viewed as a sketched method \eqref{iter} with sketches $S = (0, \ldots, 0, 1, 0, \ldots, 0)^T$, where the position of $1$ is chosen randomly at each iteration. Block Kaczmarz \eqref{blockRK} uses ${m \times s}$ sketches 
$$S = (\text{zeros}(s, \text{shift}), I(s, s), \text{zeros}(s, m - \text{shift} - s))^T,$$ 
where zeros() denotes a matrix of all zeros, $I()$ the identity, and $s$ is the block size and $\text{shift}\,=\,sz$, $z \in \{1, 2, ..., \lfloor m/s\rfloor\}$ is selected randomly at each step. For the sake of efficiency, we realize these methods by selection of rows (or row blocks) for projection rather than by the sketching procedure described above.  Gaussian Kaczmarz (non-block) uses sketches $S = \xi$, where the vectors $\xi \in \R^m$ have $N(0,1)$ independent coordinates, and BGK uses sketches $S$, where the matrices $S \in \R^{m \times s}$ have $N(0,1)$ independent coordinates.

\subsection{Dependence on the size of the block} Theorem~\ref{theor1} suggests that the per-iteration convergence rate is accelerated when the block size $s$ increases. Indeed, the decay of the relative error with the iterations is faster for bigger $s = 5, 25, 50, 100, 250, 500$ (see Figure~\ref{fig:var_block_size}). Moreover, the same trend preserves when we look at the convergence rate in time (Figure~\ref{fig:var_block_size}, right); it is worth taking larger block sizes for faster convergence. It is not specific for the matrix model, in Figure~\ref{fig:time_to_error} we plot the average number of iterations needed to achieve $1e$-$4$ relative error (average is taken over $35$ runs) for both Gaussian and coherent models; they produce quite similar shapes. Shaded regions show the variation between maximal and minimal times over $35$ runs of the algorithm. 
 
\begin{figure}
\makebox[\columnwidth]{\includegraphics[width=0.5\columnwidth]{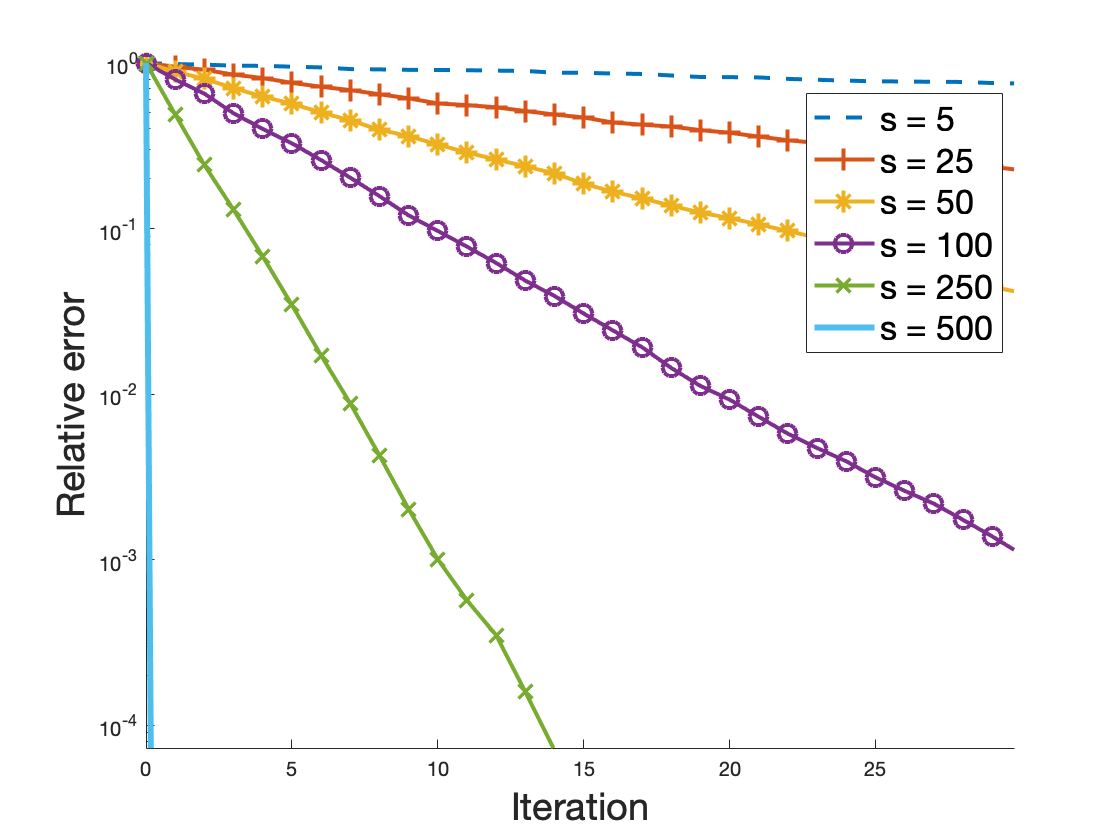} \includegraphics[width=0.5\columnwidth]{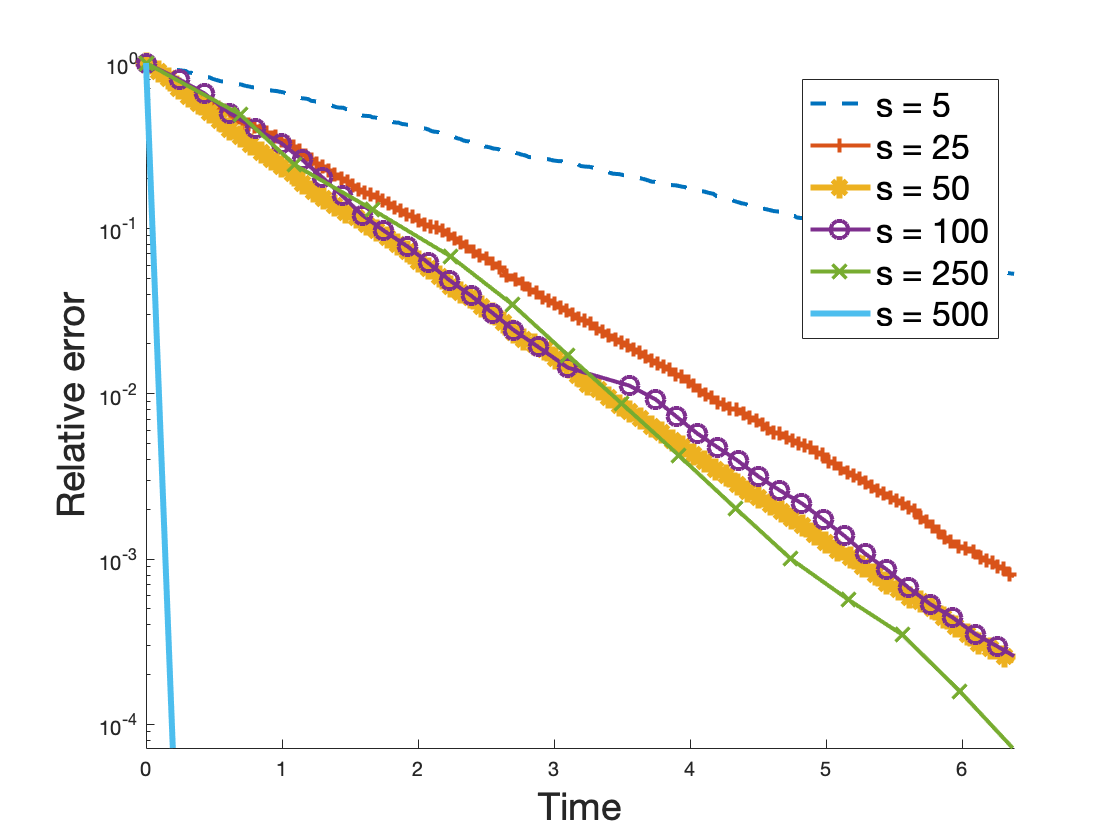}}
\caption{Gaussian model: iteration (left) and time (right) vs error for the varying block size $s$.}
\label{fig:var_block_size}
\end{figure}
\begin{figure}
\centering
\makebox[\columnwidth]{\includegraphics[width=0.5\columnwidth]{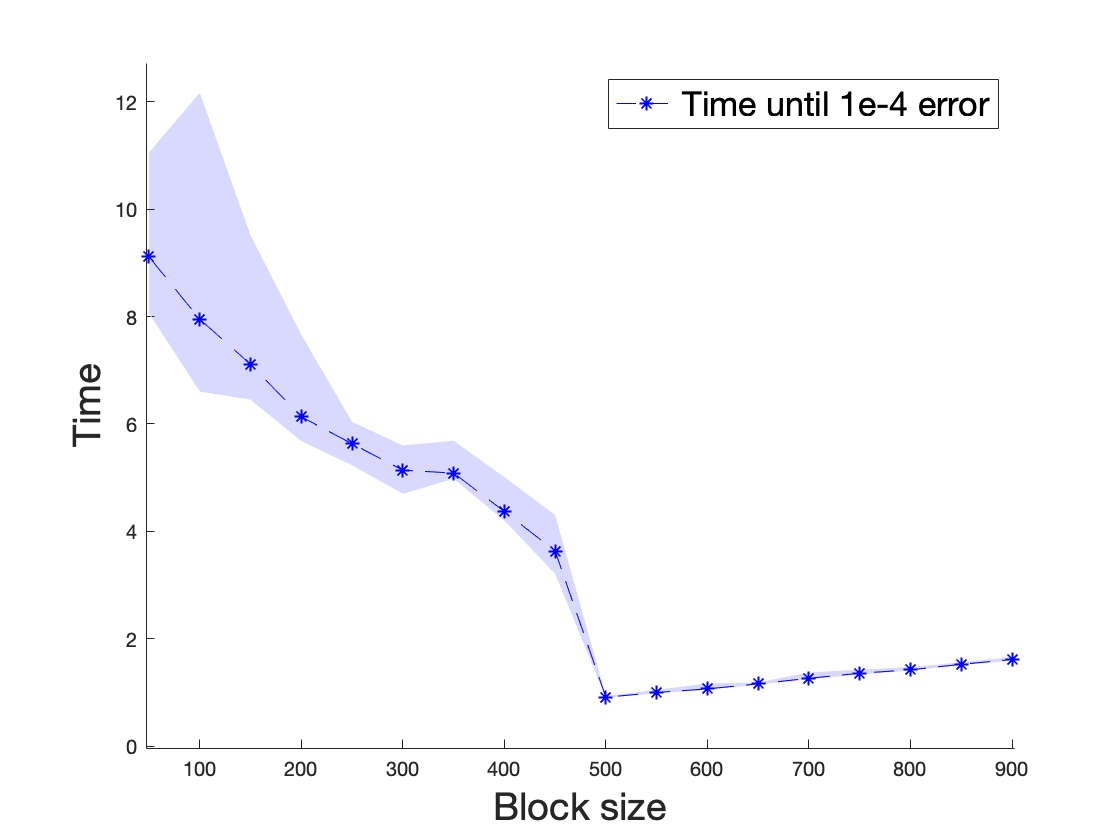} \includegraphics[width=0.5\columnwidth]{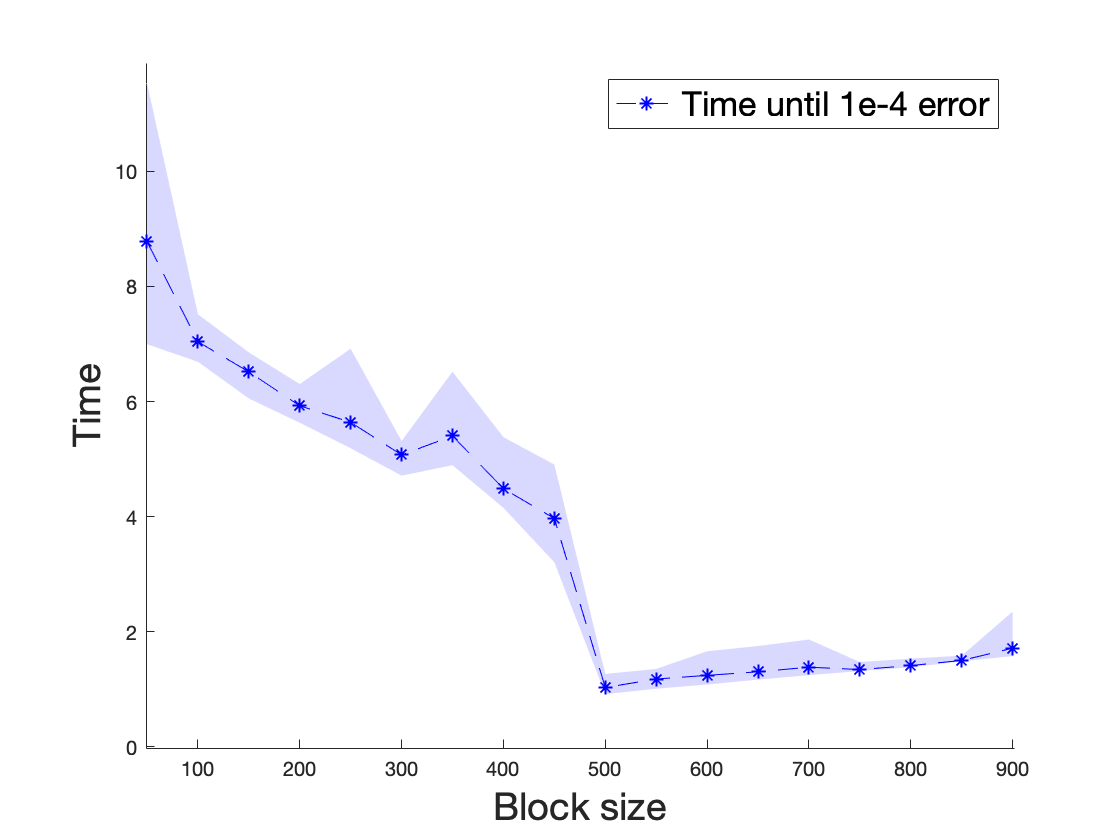}} 
\caption{Block Gaussian Kaczmarz performance on Gaussian (left) and coherent (right) models: block size vs average time until relative error reaches $1e$-$4$.}
\label{fig:time_to_error}
\end{figure}

\emph{In summary, bigger block sizes are beneficial for the BGK method. Note that since the dimension of the solution $x_k$ is $n = 500$ and $A$ has full column rank, for the block sizes larger or equal than $500$, the process converges in one iteration. So, it seems that with enough memory, but without any parallelism available, this trivial one-step version of the algorithm is the most practical.}
 
\subsection{Comparison with the other Kaczmarz methods}\label{comparison} Next, we compare rates of convergence of the Kaczmarz methods with and without Gaussian sketching. Figure~\ref{fig:method_comp_2} gives us two insights about the time performance of Gaussian and regular block Kaczmarz algorithms: (a) only in the case $s = 1$ Gaussian Kaczmarz outperforms the standard one per iteration, and (b) in time, the standard (not Gaussian) Kaczmarz method seems preferable for any block size. 

The latter observation, that Gaussian sketching typically does not seem practical comparing to the discrete randomized block sketching, is well explained by the fact that per-iteration performance is very similar in  both cases (see Figure~\ref{fig:method_comp_2}, right), but Gaussian sketching is essentially a pre-multiplication by the $s \times m$ matrix in each step, which makes iterations heavier (especially for large block sizes $s$).

\begin{figure}
\makebox[\columnwidth]{\includegraphics[width=0.5\columnwidth]{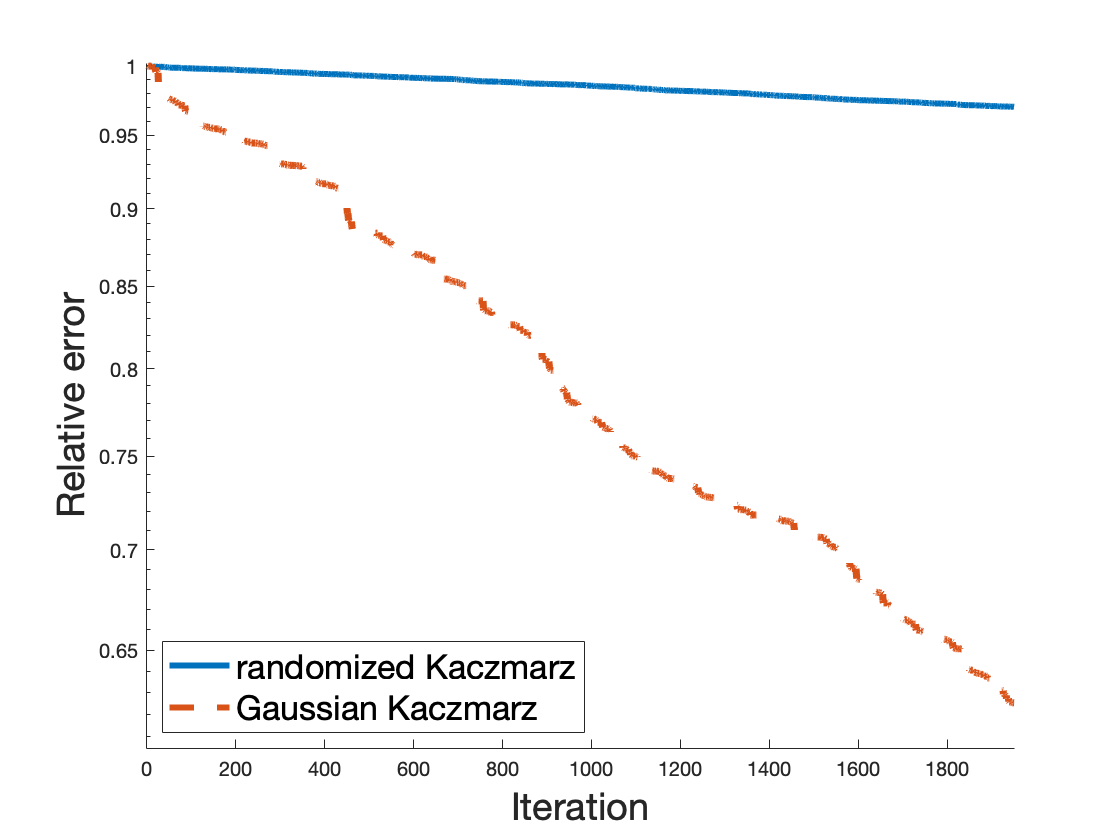} \includegraphics[width=0.5\columnwidth]{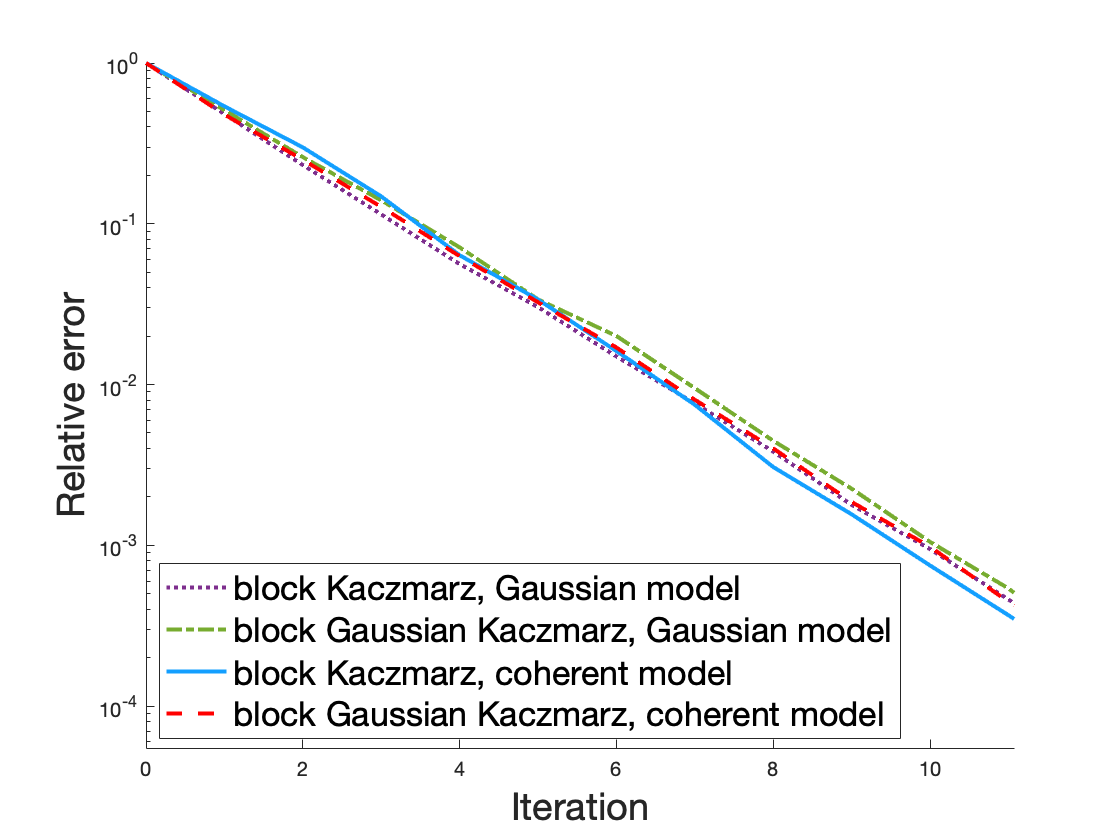} }
\caption{Left:  iteration vs relative error for $s=1$ (coherent model); right: iteration vs relative error for block methods with $s = 250$.}
\label{fig:method_comp_2}
\end{figure}
 
\begin{figure}
\makebox[\columnwidth]{\includegraphics[width=0.5\columnwidth]{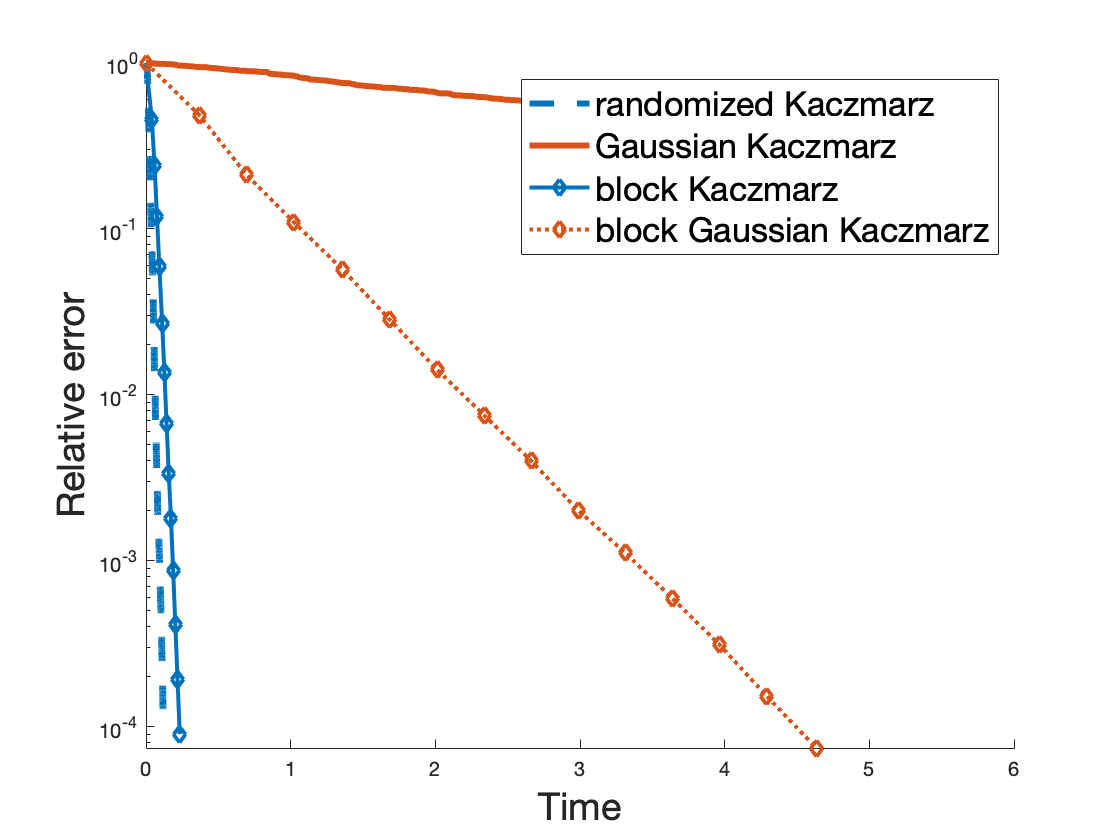} \includegraphics[width=0.5\columnwidth]{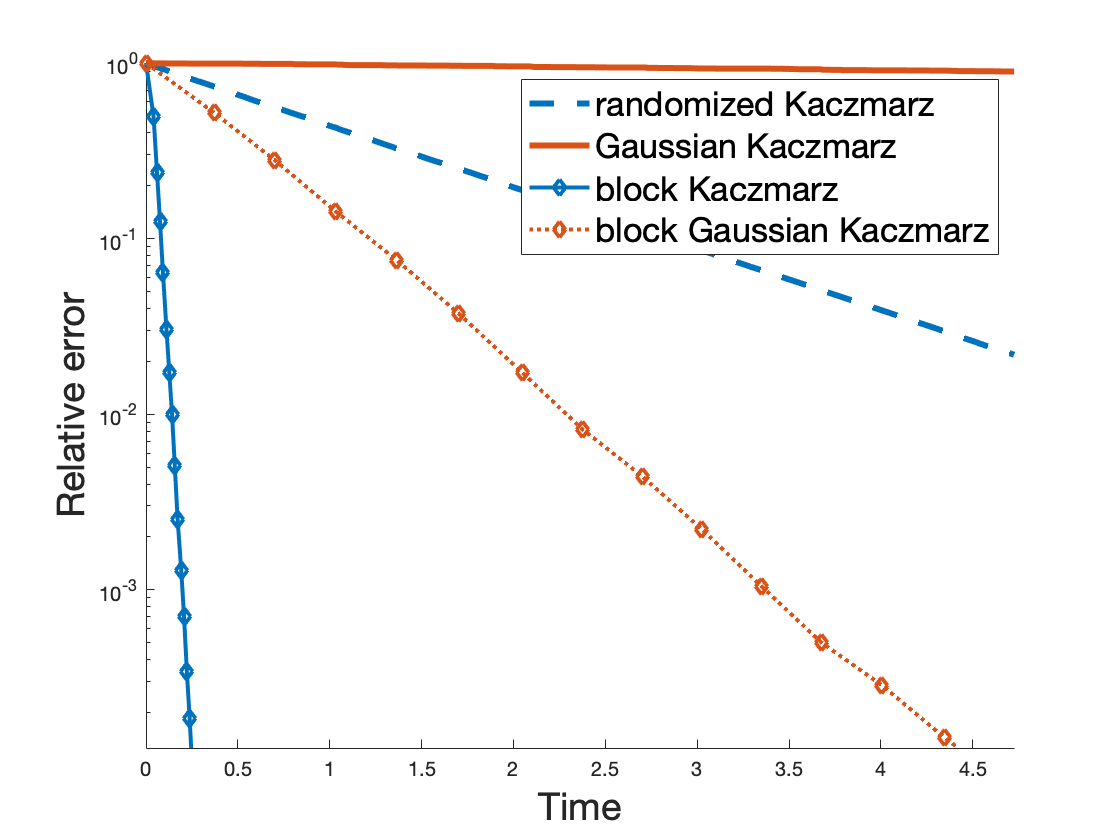} }
\caption{Gaussian (left) and coherent (right) models: decay of the relative error in time for $s=1$ (regular and Gaussian Kaczmarz) and $s=250$ (block methods). 
}
\label{fig:method_comp}
\end{figure}

The former observation, that one-dimensional Gaussian sketching is beneficial for the coherent model, can be heuristically understood geometrically in the following way. In the coherent model all the rows $A_i$ of the matrix $A$ are approximately co-linear, so iterative projections follow roughly the same direction and do not make much progress towards the solution $x_*$. Including Gaussian preprocessing, the new projection directions are $\xi_i A/\|\xi_i A\|_2$, where $\xi_1, \xi_2, \ldots$ are mutually independent Gaussian vectors, which are roughly pairwise orthogonal with high probability. Although the multiplication by $A$ creates non-trivial coherence between the projection directions, some of them still have bigger mutual angles than the angles between the row directions. Smaller scalar products $|\langle \xi_i A/\|\xi_i A\|_2, \xi_j A/\|\xi_j A\|_2\rangle|$ produce bigger convergence steps (see also Figure~\ref{fig:method_comp_2}, left). Quantitative understanding of this phenomenon is one of the interesting future directions of the current work. However, in practice, one would likely prefer to use a bigger block size and not use the Gaussian sketching step for the sake of better time performance (see Figure~\ref{fig:method_comp}). 

Although it may seem there is never a case when the BGK method is most practical, to be fair, one can actually construct examples where it is.  The ``mixed model" is a tall $50000\times 500$ matrix $A$ that contains  $500$ random independent standard normal rows, and all other rows are identical (in our experiments, repetitions of the first Gaussian row). Then, to solve the system $Ax = b$, the iterative method needs to ``find" $500$ different rows. Selection of random blocks of reasonable size $s \le n = 500$ is very inefficient in this search. In contrast, Gaussian sketching forces each iteration to use the information about the whole matrix $A$, thereby ``smoothing" the mixed structure of the matrix. As a result, we observe enough per-iteration gain so that even with the heavier computation, the Gaussian iterations converge faster (see Figure~\ref{fig:mixed}).

\begin{figure}
\makebox[\columnwidth]{\includegraphics[width=0.5\columnwidth]{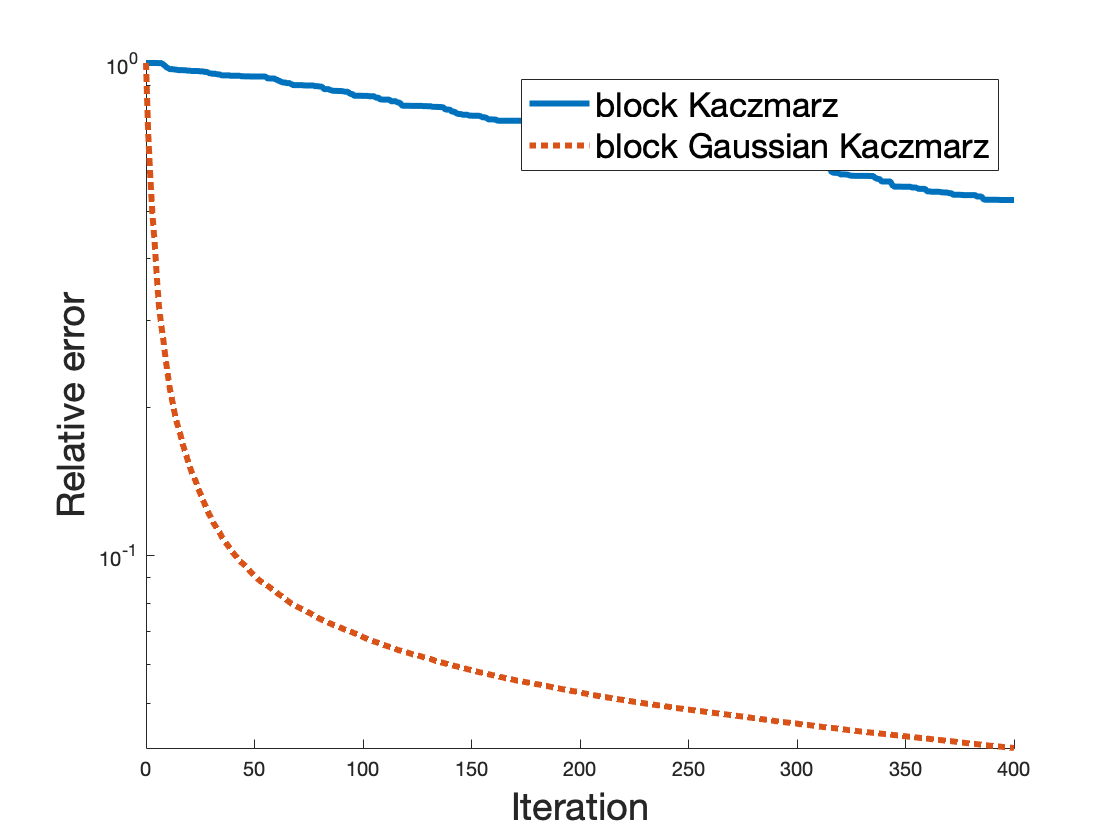} \includegraphics[width=0.5\columnwidth]{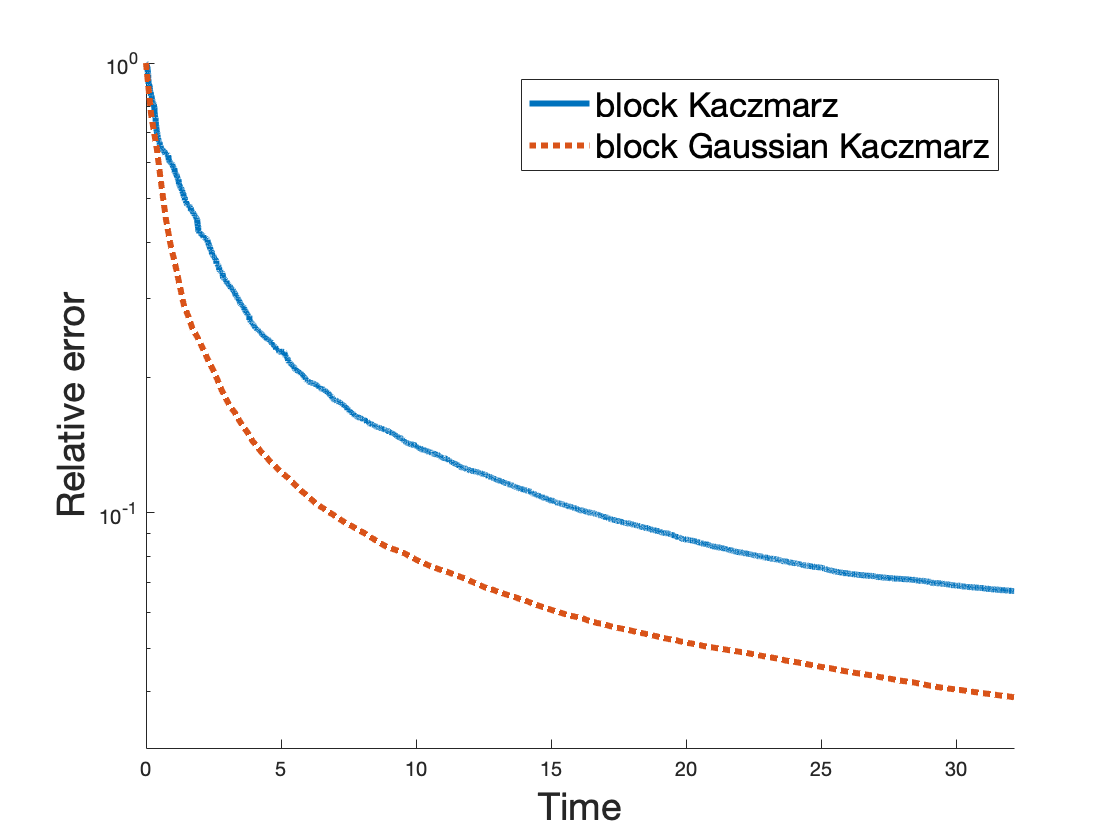}}
\caption{Mixed model: iteration (left) and time (right) vs relative error for the BGK with $s=100$.}
\label{fig:mixed}
\end{figure}

\begin{figure}
\makebox[\columnwidth]{\includegraphics[width=0.5\columnwidth]{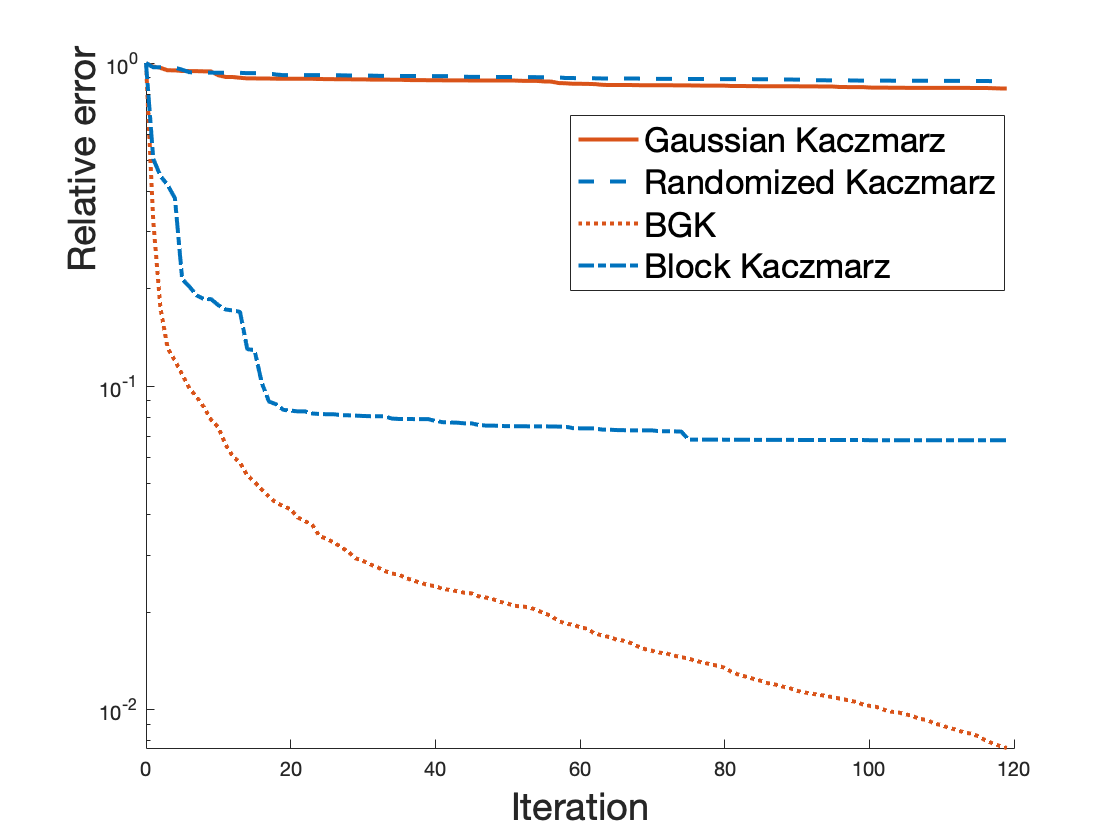}\includegraphics[width=0.5\columnwidth]{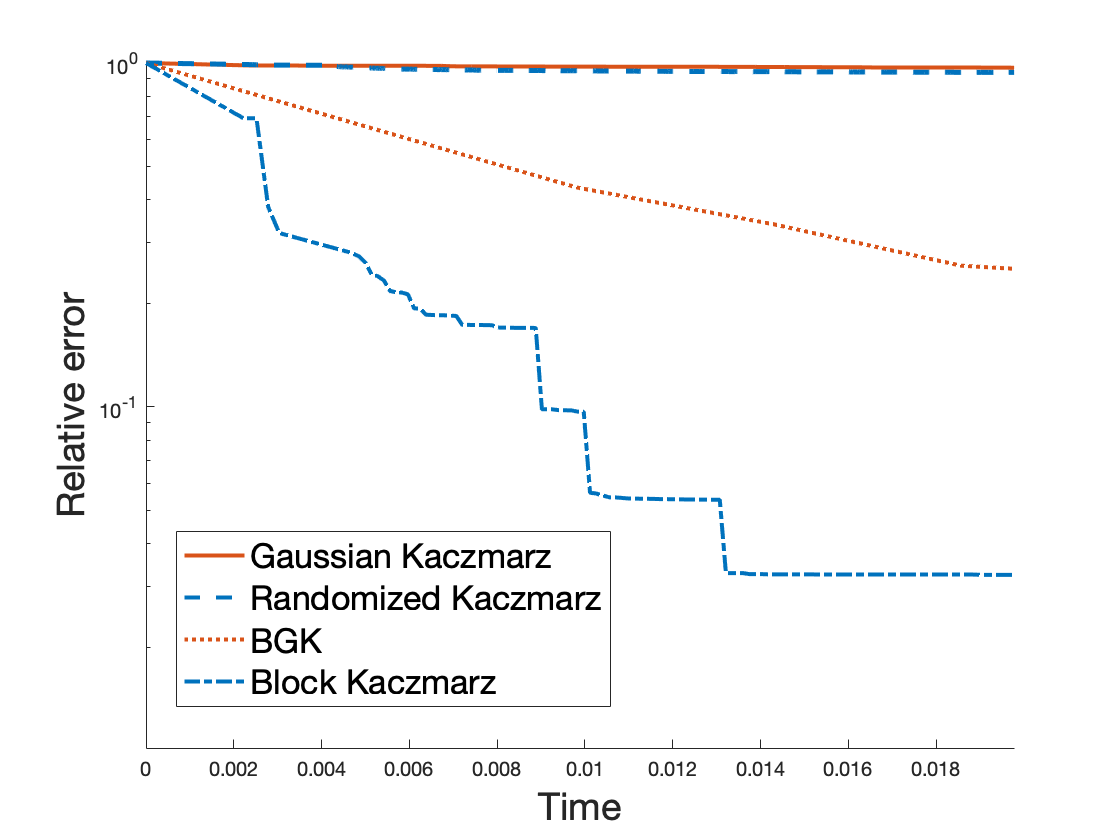}}
\caption{COVTYPE $10000\times 50$ dataset: iteration (left) and time (right) vs relative error for the BGK with $s=30$.}
\label{fig:real_experiments_covtype}
\end{figure}

\begin{figure}
\makebox[\columnwidth]{\includegraphics[width=0.5\columnwidth]{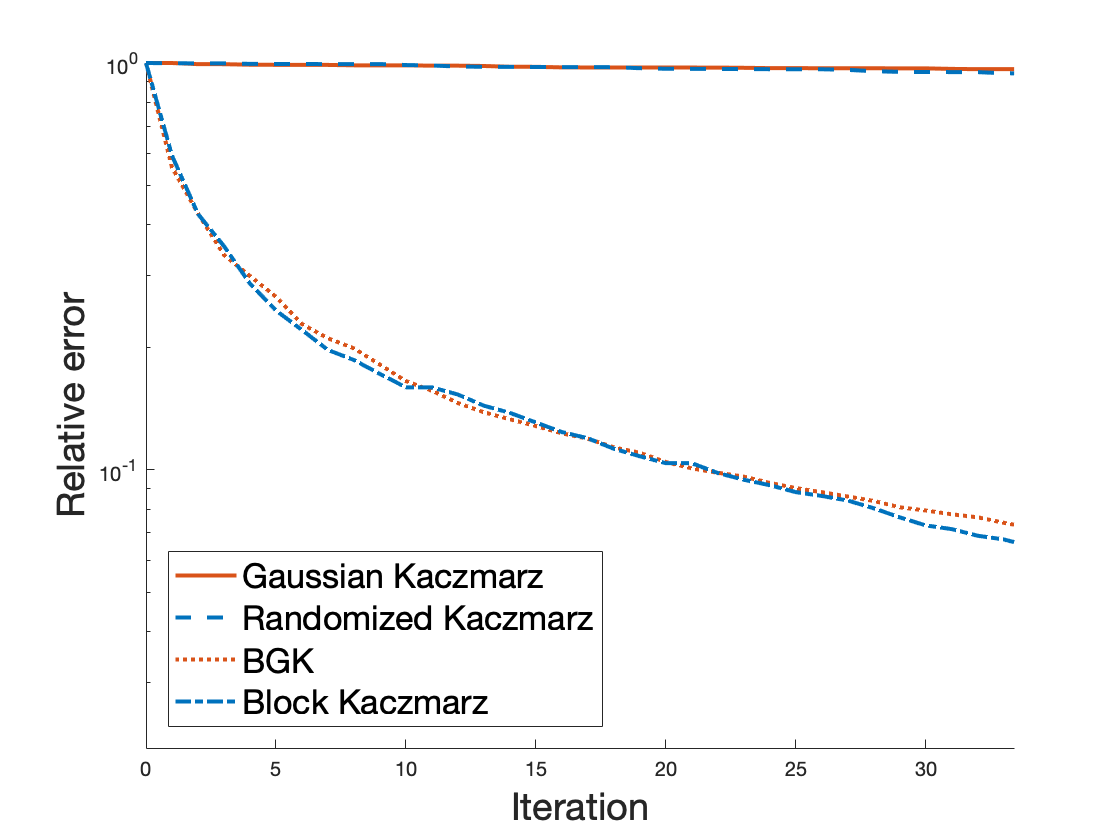}\includegraphics[width=0.5\columnwidth]{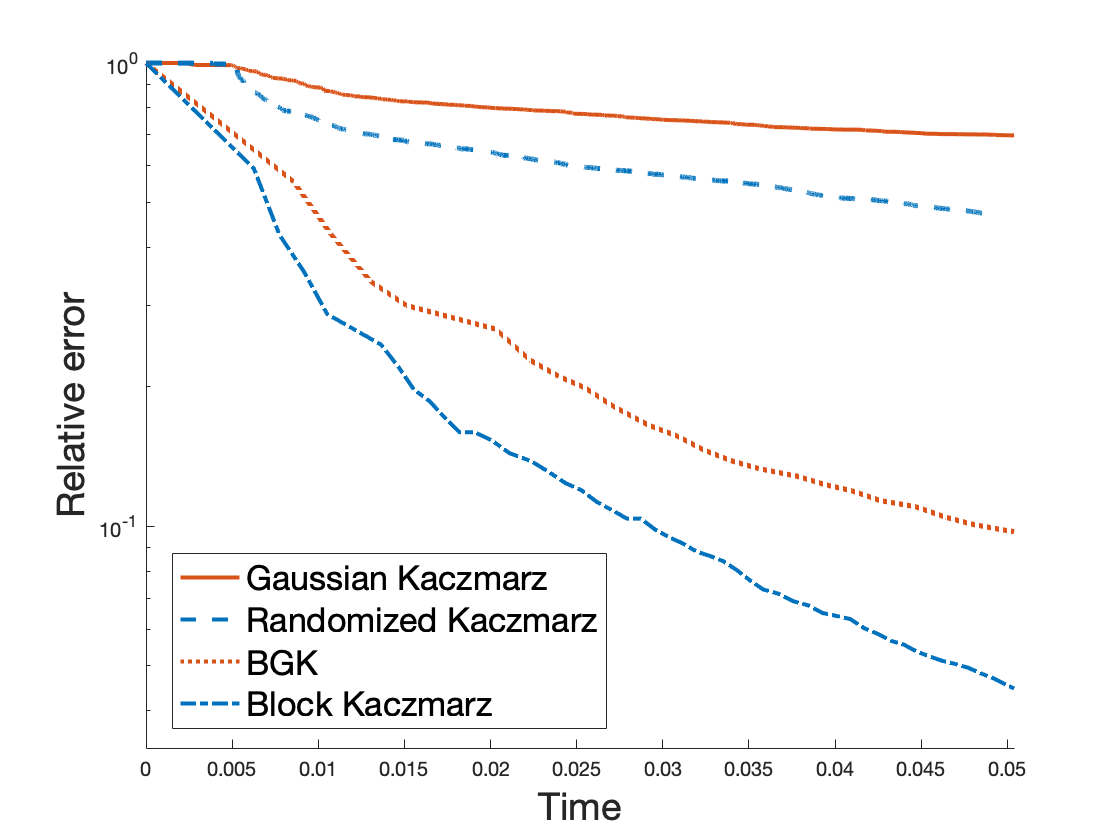}}
\caption{GAS $1000\times 128$ dataset: iteration (left) and time (right) vs relative error for the BGK with $s=50$.}
\label{fig:real_experiments_gas}
\end{figure}

\emph{In summary, Gaussian sketching is probably not the best way to proceed solving a general system of equations, either coherent or not, since it achieves exactly the same per-iteration performance as regular block Kaczmarz (so one would prefer the lightest possible iteration presented by selection of a random block rather than a matrix multiplication). On some special models, when the structure of $A$ leads to very poor conditioning, Gaussian sketching can in some sense automate this search due to its ``smoothing" property.}

\subsection{Experiments with real world datasets}

As we can see in Figure \ref{fig:method_comp_2} right, for both general artificial models (Gaussian and coherent) Block Kaczmarz and BGK show equally well per iteration progress. Given this, it does not make sense to use slower Gaussian sketching in practice. In order to see the advantage of Gaussian sketched iterations we had to construct a very special model (Figure~\ref{fig:mixed}). However, we also observe this phenomenon sometimes on real world datasets. For example, COVTYPE dataset also has this property: Gaussian block sketched iterations are more efficient than the iterations of a standard block method (Figure~\ref{fig:real_experiments_covtype} left). Unlike in the case of the artificial ``mixed" model ( Figure~\ref{fig:mixed}), here the gain in per iteration efficiency is not enough to have better convergence rate in time (Figure~\ref{fig:real_experiments_covtype} right). Only given some more efficient distributed realization of sketching could we expect BGK to be the most time efficient on the datasets like COVTYPE.

Moreover, there are datasets like GAS where there is no per iteration gain due to Gaussian sketches (see \ref{fig:real_experiments_gas}). Both  mentioned  datasets  are  taken  from  the  UCI  repository \cite{Dua:2019}.

\emph{In summary, there are real world datasets on which it is beneficial to use Gaussian block sketching to make iterations converge faster. But it is likely that one needs to use a faster way  to apply sketches (speeding up matrix multiplication and potentially distributing its computation between several cores) in order to preserve this advantage in time evaluation.}

\subsection{Finite number of samples} Now, let us consider the finite form of sketching presented in Theorem~\ref{theor3}, when we pre-select a set of Gaussian matrices $\mathcal{S}$ and take sketch matrices from it. Although our theoretical analysis requires the cardinality of $\mathcal{S}$ to be at least $C m^2 \log m$, the numerical experiments show that in practice a cardinality much smaller than $m$ (on the order of $m/s$) is enough to reproduce the per-iteration convergence rate of the original BGK method. However, for smaller collections it might take several extra iterations to converge, and with very small collections $\mathcal{S}$, the method stops converging far from the solution (see Figure~\ref{fig:from_collect}).

\begin{figure}
\makebox[\columnwidth]{\includegraphics[width=0.33\columnwidth]{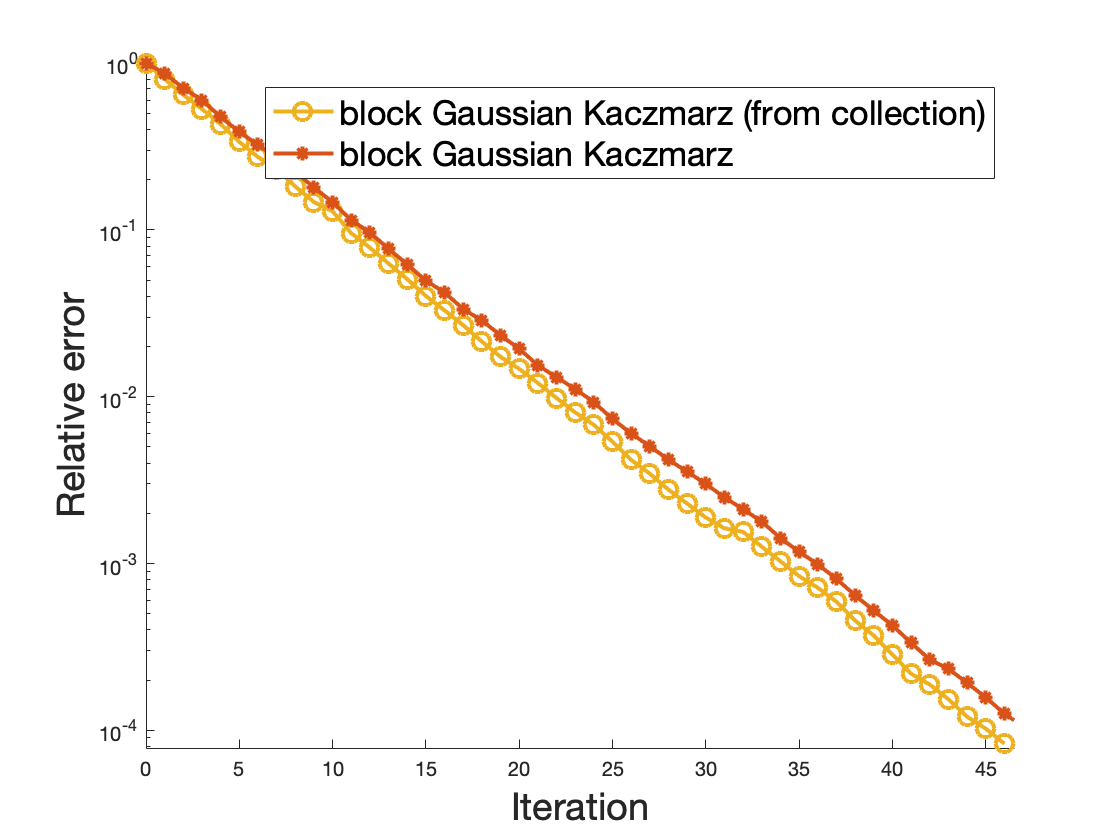} \includegraphics[width=0.33\columnwidth]{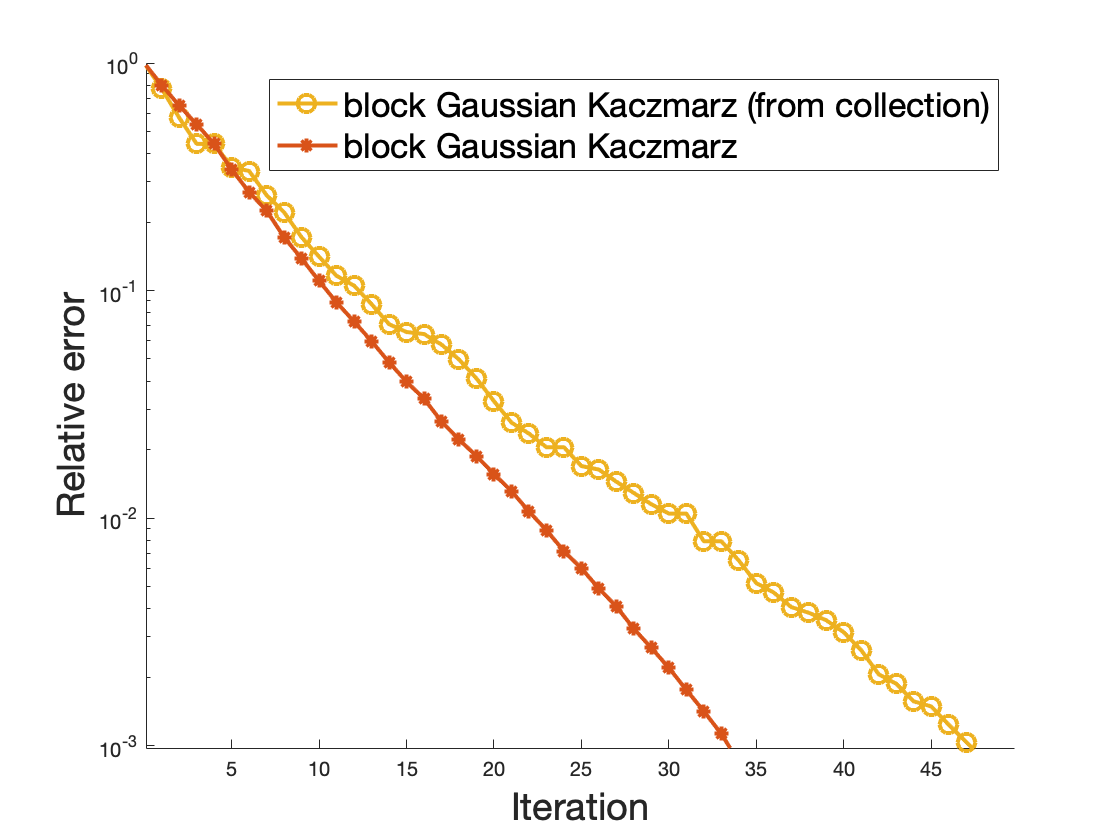}\includegraphics[width=0.33\columnwidth]{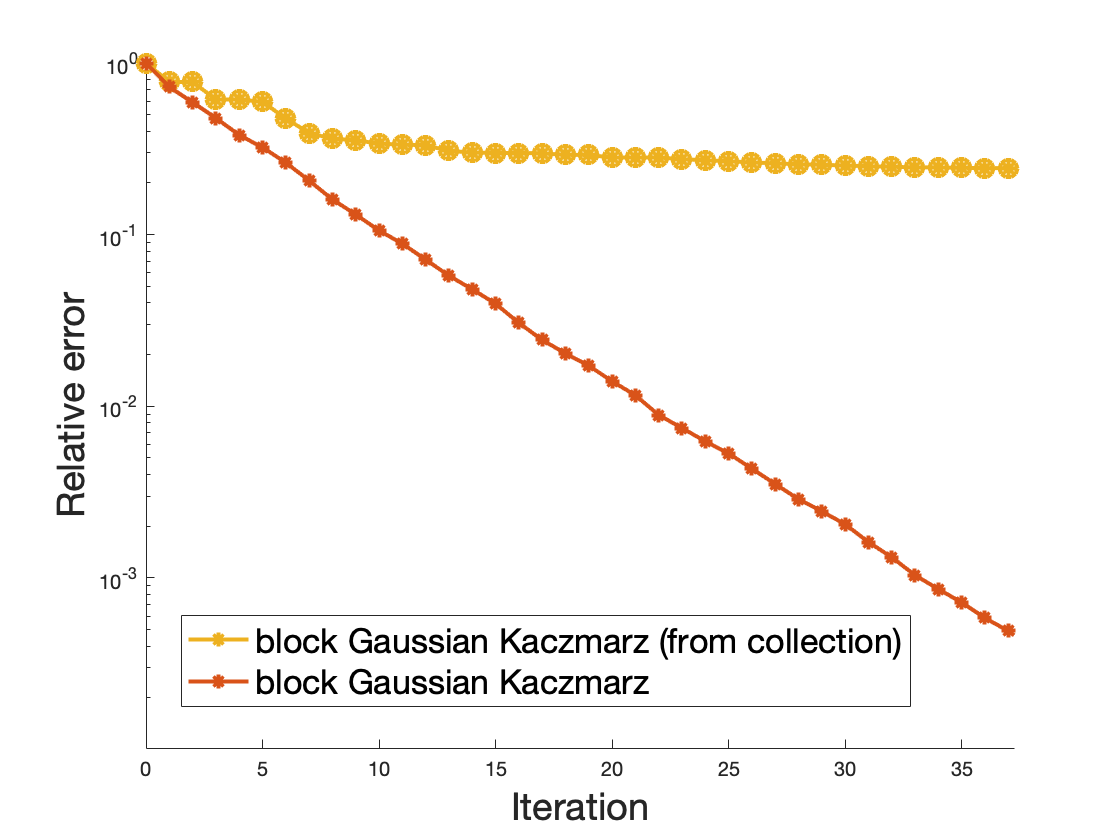}}
\caption{Red: BGK selecting new Gaussian sketch each time; yellow: same method selecting sketches from the finite collection $\mathcal{S}$.  Gaussian matrix model $A \sim randn(5000, 500)$, block size $s = 100$. Collection sizes: (left) $|\mathcal{S}| = 200$ -- convergence is as fast as it can be; (middle) $|\mathcal{S}| = 25$ -- several more iterations are needed to reach $1e$-$3$ relative error; (right)  $|\mathcal{S}| = 5$ -- collection is too small and convergence fails.}
\label{fig:from_collect}
\end{figure}

\emph{To summarize, one provably does not have to have infinite resources to carry out Gaussian sketching for the BGK method, but in practice it is usually inconvenient to pre-generate (and store) all samples we might need in the future, for both memory and computational reasons. It is an interesting mathematical question whether one can sharpen the size of $\mathcal{S}$ theoretically.}

\subsection{Inconsistent Gaussian Kaczmarz}\label{noisy}
In our final set of experiments, we consider the case when the system is inconsistent $Ax_* = b + e$. Then, Theorem~\ref{theor4} reveals a new dependence of the performance on the size of the sketch $s$. Indeed, the next iteration's relative error is the sum of two terms: the first one depending on the distance to the previous iteration, and the second one depending on the level of the noise $\|e\|_2$ (see \eqref{c_rate_noisy}). The first term is the same as in the noiseless case, and it gets smaller with increasing $s$, but the second ``convergence horizon" hits a singularity at $s = n$. So, taking sketch size $n$ for a one-step solution, like in the noiseless case, is a very poor choice now (see Figure~\ref{fig:noisy_blocks}, left).

Depending on the level of noise, the width of this singularity region changes (see Figure~\ref{fig:noisy_blocks}, right). Although it still seems that taking  large $s > n$ a faster solution, without the prior knowledge about the size of the singularity region, it might be safer to take smaller sketches $s \ll n$. Figure~\ref{fig:noisy_blocks} also suggests that the optimal $s$ in the range $s \in [1, n]$ also seems to depend on the noise level.

\begin{figure}
\centering
\makebox[\columnwidth]{\includegraphics[width=0.5\columnwidth]{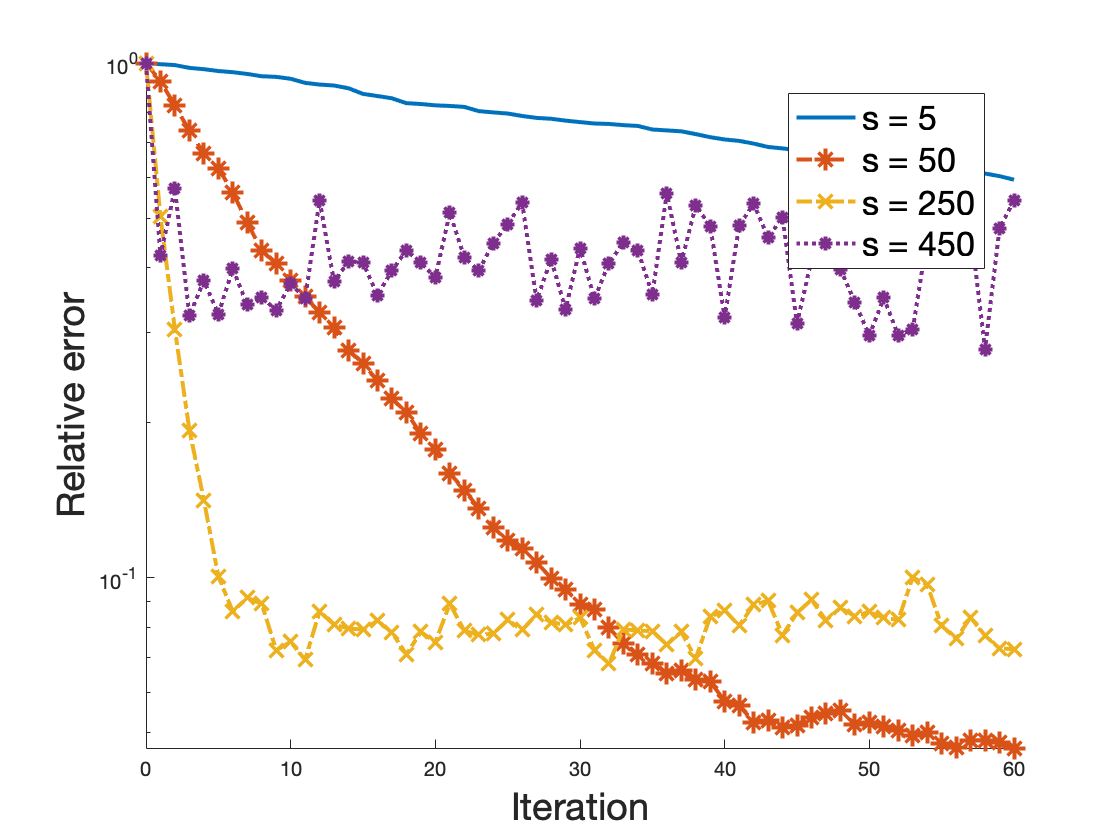} \includegraphics[width=0.5\columnwidth]{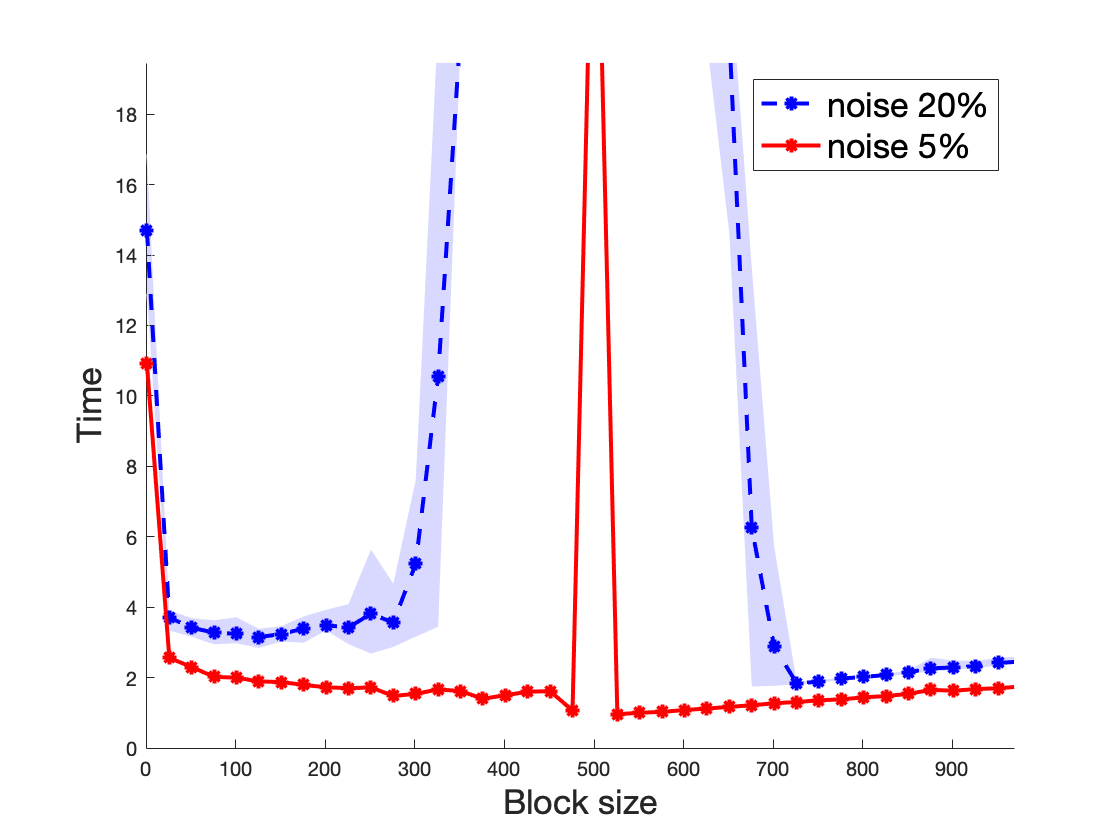}}
\caption{Left: Gaussian model with with $20\%$ Gaussian noise: dependence on the block size $s$. Right: time to reach $1e$-$1$ error for various block sizes, $\|e\|_2$ is normalized to be $20\%$ ($5\%$) of $\|b\|_2$}
\label{fig:noisy_blocks}
\end{figure}
 
Finally, Figure~\ref{fig:spiky_noise} shows the rate of per-iteration convergence of the Gaussian and coherent systems in the presence of sparse spiky noise (in the Gaussian case, it was generated as $50$ random spikes of  magnitude $50$, and in the coherent case it was generated as as $10$ spikes of magnitude $25$). We can see that the Gaussian system is much more robust to the same noise level. We also see that without Gaussian sketching the convergence rate experiences random spikes at some iterations (when the current block contains the index corresponding to one of ten ``spikes" in the noise). This incurs significant variability of the error at each specific iteration. Although Gaussian sketching makes the resulting convergence horizon a bit higher, it completely resolves the variability problem. Informally, this happens since the Gaussian sketch incorporates all the information about the error term into each iteration, ``smoothing" the spikes in the noise. See also Remark~\ref{high_prob_remark}.

\begin{figure}
\centering
\makebox[\columnwidth]{\includegraphics[width=0.5\columnwidth]{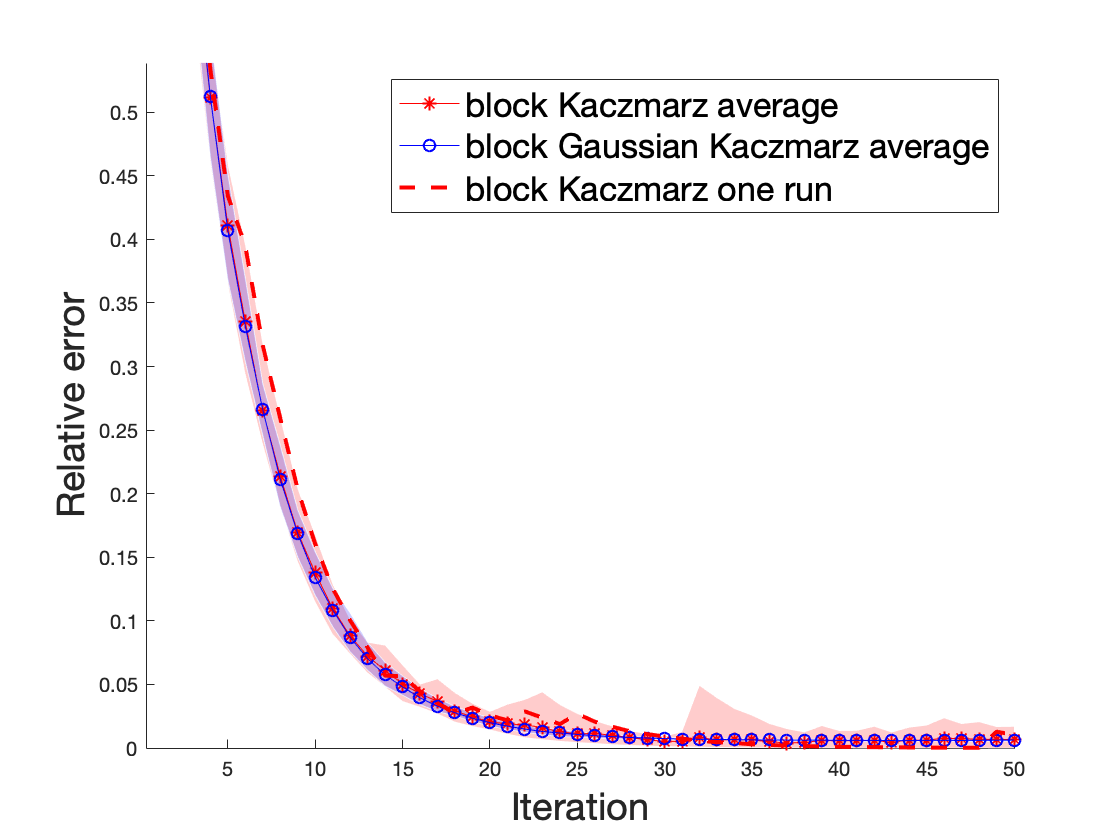} \includegraphics[width=0.5\columnwidth]{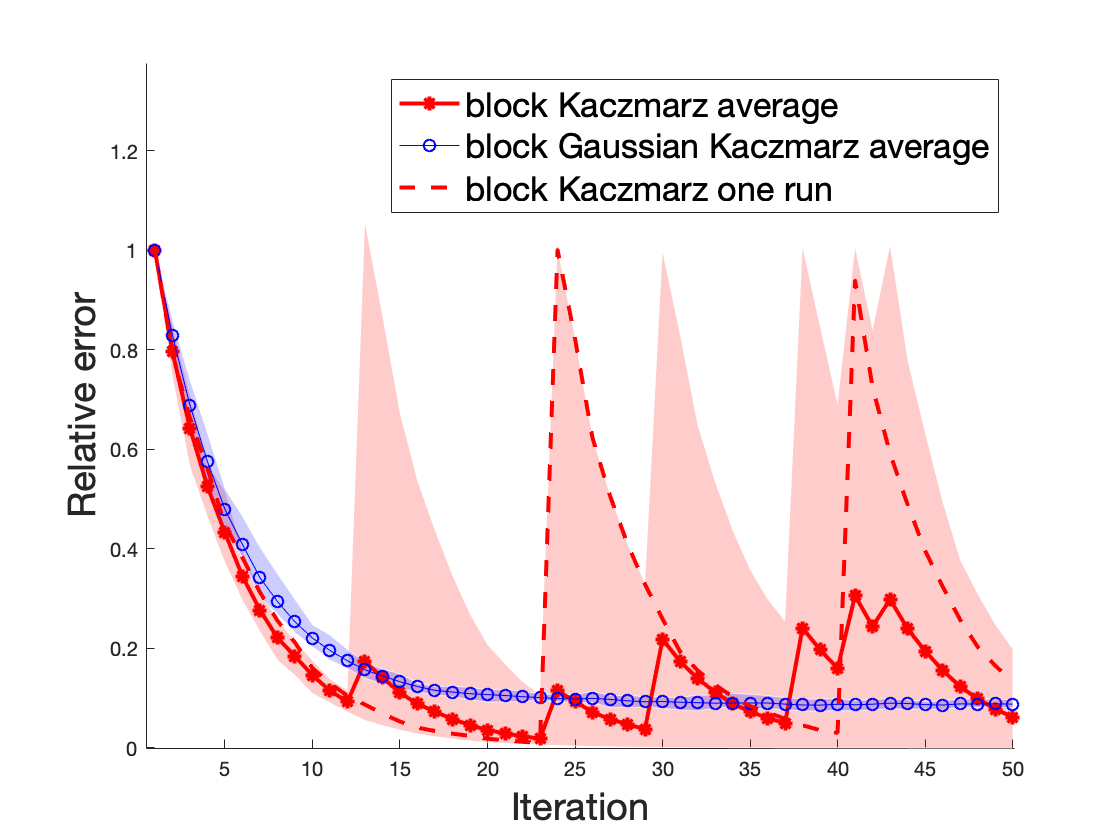}}
\caption{Gaussian (left) and coherent (right) models with spiky noise: iteration average relative error and its range over $10$ runs; $s = 100$}
\label{fig:spiky_noise}
\end{figure}

\emph{In summary, for inconsistent systems one must be careful with the choice of the sketch size, since choosing $s \sim n$ makes the error term very large, and the method may not converge. In the inconsistent case, Gaussian sketching once again does not seem to improve the convergence rate over the standard block method with the same size $s$; however it can reduce the variability of the iterates, making the convergence more predictable. So, Gaussian sketching might be used in noisy systems for the sake of robustness, rather than efficiency.}

\section{Conclusions and future directions} \label{conclusions}
To the best of our knowledge, our paper presents the first theoretical analysis of the exponential convergence properties of the BGK method for arbitrary sketch size (Theorems~\ref{theor1} and \ref{theor2}). For the one-dimensional sketches, when $s=1$, Theorem~\ref{theor2} recovers the expected exponential convergence rate with the factor $1 - c\sigma_{min}^2(A)/ \|A\|_F^2$, recovering the results of Gower and Richt\'arik \cite{GowRic}. 

Unlike the analysis of the randomized block Kaczmarz method, presented in \cite{NeeTro}, our convergence guarantees for the Gaussian version do not require any additional structural assumptions on the matrix (such as regularized rows), or any non-trivial preprocessing (to find a ``good" partition of the rows). Additionally, our bounds allow to trace the dependence of the convergence rate on the size of the sketch.

We also analyze a finite collection approach, that allows one to avoid a  potentially infinite generation of the Gaussian sketch matrices.  Indeed, we prove (Theorem~\ref{theor3}) that with high probability, a random collection of Gaussian matrices will provide the same rate of convergence bound as the original random sampling approach. Finally, we give the theoretical convergence analysis for the case of the inconsistent systems of equations (Theorems~\ref{theor4} and \ref{theor5} ).

Our numerical experiments confirm the theoretical claims about the dependence on the sketch size in the consistent and inconsistent cases, and outperform the theoretically required size of the finite sampling collection of sketches. In terms of the convergence time, our experiments support  the observation that Gaussian methods almost always require more computation to reach the desired solution as their non-sketching standard counterparts. We found only a special set of examples when the block Gaussian method may be practically preferable to the standard randomized block method  due to the ``smoothing" effect. Besides these carefully chosen models, perhaps the only other advantage the Gaussian methods exhibit is the unsurprising reduction in variance of the iterates.  Provably quantifying this may be another interesting mathematical problem.  

Among potential further directions of this work is theoretical justification that the sample set $|\mathcal{S}| \sim O(m)$ is enough for the result like Theorem~\ref{main3} to hold. It would be also interesting to get the convergence rate estimate that explicitly demonstrates the gain of the Gaussian sketching in the coherent case (as we see in Figure 3, per iteration rates of all block versions are the same for the incoherent and coherent systems, but the theoretical rate depends on $\sigma_{min}(A)$, and it is much worse in the latter case). So far the only progress in this direction was made in \cite{NeeWar} for the two dimensional subspace projection.
  
 It is natural to consider other sketches than Gaussian (in particular, sparsified Gaussian or Fourier sketches that allow faster multiplication). A slight modification of our analysis in Theorem~\ref{main2} would extend the analysis to any sub-gaussian sketch matrices when the variance of each entry is bounded from below away from zero (which excludes very sparse sketch matrices). However, a systematic study of the various sketches is still among the further directions of the current work. Another potential extension could be to solve inexact linear systems (without the assumption about full rank column space of $A$), see also \cite{loizou2019convergence} for the discussion of this framework.

\bibliography{liza-bib-copy}   % name your BibTeX data base

\end{document}